\documentclass[a4paper, 10 pt, reqno]{article}

\usepackage{amsfonts, amssymb, amsmath, eucal, amsthm}
\usepackage[all]{xy} 

\newtheorem{lemma}{Lemma}
\newtheorem{theorem}[lemma]{Theorem}
\newtheorem*{theorem*}{Theorem}
\newtheorem{proposition}[lemma]{Proposition}
\newtheorem{corollary}[lemma]{Corollary}

\newtheorem*{TheoremA}{Theorem A}
\newtheorem*{TheoremB}{Theorem B}
\newtheorem*{TheoremC}{Theorem C}
\newtheorem*{TheoremD}{Theorem D}
\theoremstyle{definition}
\newtheorem{definition}[lemma]{Definition}
\newtheorem{note}[lemma]{Note}

\newtheorem*{corollary*}{Corollary}

\numberwithin{lemma}{section}

\newcommand{\ev}{\mathrm{ev}}
\newcommand{\Th}{\mathrm{Th}}
\newcommand{\CP}{\mathbb{C}\mathrm{P}}
\DeclareMathOperator{\rank}{rank}

\newcommand{\StringLieGroups}{StringLieGroups}
\newcommand{\ChasSullivan}{ChasSullivan}
\newcommand{\CohenJones}{MR1942249}
\newcommand{\CohenJonesYan}{MR2039760}
\newcommand{\CohenKlein}{MR2475820}
\newcommand{\EvenBott}{MR0087035}
\newcommand{\MilnorAndStasheff}{MR0440554}
\newcommand{\Spanier}{MR666554}
\newcommand{\TamanoiCap}{TamanoiCap}
\newcommand{\Menichi}{MR2466078}
\newcommand{\TamanoiBV}{MR2211159}
\newcommand{\Yang}{Yang}
\newcommand{\Vaintrob}{Vaintrob}
\newcommand{\Tradler}{MR2498354}

\newcommand{\FelixThomas}{MR2415345}

\title{String topology for complex projective spaces}
\author{Richard Hepworth\footnote{The author is supported by E.P.S.R.C.~Postdoctoral Research Fellowship EP/D066980.}\\  Department of Pure Mathematics\\ University of Sheffield}
\date{}
\begin{document}
\maketitle

\begin{abstract}
\noindent In 1999 Chas and Sullivan showed that the homology of the free loop space of an oriented manifold admits the structure of a Batalin-Vilkovisky algebra.  In this paper we give a complete description of this Batalin-Vilkovisky algebra for complex projective spaces.  This builds on a description of the ring structure that is due to Cohen, Jones and Yan.  In the course of the proof we establish several new general results.  These include a description of how symmetries of a manifold can be used to understand its string topology, and a relationship between characteristic classes and circle actions on sphere bundles.
\end{abstract}

\section{Introduction}

\subsection{Background and results}\label{ResultSubsection}
Let $M$ be a closed oriented smooth manifold of dimension $m$ and let $LM$ denote the space of free loops in $M$.  Chas and Sullivan \cite{\ChasSullivan}
described how the homology of $LM$ can be endowed with a \emph{loop product}
\[H_\ast(LM)\otimes H_\ast(LM)\to H_{\ast -m}(LM)\]
and a \emph{BV operator}
\[\Delta\colon H_\ast(LM)\to H_{\ast+1}(LM)\]
that together make $\mathbb{H}_\ast(LM):=H_{\ast +m}(LM)$ into a Batalin-Vilkovisky algebra, or BV algebra.  This means that the loop product is associative and graded commutative, that $\Delta^2=0$, and that the Batalin-Vilkovisky identity
\begin{align}\Delta(x\cdot y\cdot z)&=
\nonumber \Delta(x\cdot y)\cdot z+(-1)^{|x|}x\cdot\Delta(y\cdot z)+(-1)^{(|x|-1)|y|}y\cdot\Delta(x\cdot z)\\
\label{BVIdentityEquation}&-(\Delta x)\cdot y\cdot z-(-1)^{|x|}x\cdot (\Delta y)\cdot z-(-1)^{|x|+|y|}x\cdot y\cdot (\Delta z)\end{align}
holds.
We will refer to $\mathbb{H}_\ast(LM)$ as the \emph{string topology BV algebra of $M$}.

The string topology BV algebra has by now been computed for several classes of manifolds.  These are the complex Stiefel manifolds \cite{\TamanoiBV} and  spheres \cite{\Menichi}, where coefficients were taken in the integers; the complex and quaternionic projective spaces \cite{\Yang} and surfaces of genus $g>1$ \cite{\Vaintrob}, where coefficients were taken in the rational numbers;  and the compact Lie groups \cite{\StringLieGroups}, where coefficients were taken in any commutative ring.  The aim of this paper is to describe the string topology BV algebra for $\CP^n$, $n\geqslant 1$, with coefficients in $\mathbb{Z}$.   The ring structure has already been computed by Cohen, Jones and Yan:

\begin{theorem*}[Cohen, Jones, Yan \cite{\CohenJonesYan}]
As a ring,
\begin{equation}\label{RingDescription}\mathbb{H}_\ast(L\CP^n;\mathbb{Z})=\frac{\Lambda_\mathbb{Z}[w]\otimes\mathbb{Z}[c,v]}{\langle c^{n+1}, (n+1)c^n\cdot v, w\cdot c^n\rangle},\end{equation}
where $|w|={-1}$, $|c|={-2}$ and $|v|={2n}$.
\end{theorem*}

This theorem allows for some freedom in the choice of generators $c$, $w$ and $v$.  Any or all of the generators may be negated, and any multiple of $c^n\cdot v^2$ may be added to $v$, without altering the result.  We complete the description of $\mathbb{H}_\ast(L\CP^n)$ with the following theorem.

\begin{TheoremA}
After making an appropriate choice of $c$, $w$ and $v$, the BV operator $\Delta$ on $\mathbb{H}_\ast(L\CP^n;\mathbb{Z})$ is given by $\Delta(c^p\cdot v^q)=0$ and
\[\Delta(c^p\cdot w\cdot v^q)= \left[(n-p)+q(n+1)\right] c^p\cdot v^q + (q+1)\binom{n+1}{2}c^{n+p}\cdot v^{q+1}\]
for all $p,q\geqslant 0$.  Note that $\binom{n+1}{2}c^{n+p}\cdot v^{q+1}$ vanishes when $p>0$ and when $n$ is even, and that it has order $2$ otherwise.
\end{TheoremA}

The constants $(n+1)$ and $\binom{n+1}{2}$ appearing in the description of $\Delta$ arise from characteristic classes of $\CP^n$.  To be precise, they are the quantities $\langle c_n(T\CP^n),[\CP^n]\rangle$ and $\langle c_{n-1}(T\CP^n),[\CP^{n-1}]\rangle$ respectively.  It would be interesting to see whether there is a general relationship between the Chern classes of a complex manifold $M$ and the BV algebra $\mathbb{H}_\ast(LM)$.  

The graded ring underlying any Batalin-Vilkovisky algebra, when equipped with the bracket operation $\{a,b\}=(-1)^{|a|}\Delta(a\cdot b)-(-1)^{|a|}(\Delta a)\cdot b - a\cdot (\Delta b)$, becomes a \emph{Gerstenhaber algebra}.  This means in particular that the bracket is a derivation in both variables and that $\{a,b\}=-(-1)^{(|a|+1)(|b|+1)}\{b,a\}$.    As an immediate consequence of the two theorems above we have the following description of the Gerstenhaber algebra $\mathbb{H}_\ast(L\CP^n)$:
\begin{corollary*}
As a Gerstenhaber algebra, $\mathbb{H}_\ast(L\CP^n;\mathbb{Z})$ is determined by \eqref{RingDescription}, by
\begin{align*}
\{c,w\}=-\{w,c\}=&-c,\qquad\\
\{v,w\}=-\{w,v\}=&(n+1)v+\binom{n+1}{2}c^n\cdot v^2,
\end{align*}
and by the fact that the bracket vanishes on all other pairs among the generators $c$, $w$, $v$.
\end{corollary*}

\subsection{Relationship with previous computations}

Theorem A, combined with the theorem of Cohen, Jones and Yan, allows us to recover two of the existing computations mentioned in the last subsection.

\begin{theorem*}[Menichi \cite{\Menichi}]
The string topology BV algebra of $S^2$ is given by
\[\mathbb{H}_\ast(LS^2;\mathbb{Z})=\frac{\Lambda_\mathbb{Z}[b]\otimes\mathbb{Z}[a,v]}{\langle a^2,a\cdot b,2a\cdot v\rangle}\]
with $|b|=-1$, $|a|=-2$ and $|v|=2$.  For all $k\geqslant 0$ we have $\Delta(v^k)=0$, $\Delta(a\cdot v^k)=0$ and $\Delta(b\cdot v^k)=(2k+1)v^k+a\cdot v^{k+1}$.
\end{theorem*}
\begin{proof}
Since $S^2=\CP^1$, we may take $n=1$ in 
the two theorems of \S\ref{ResultSubsection} to find that $\mathbb{H}_\ast(LS^2;\mathbb{Z})$ is equal to ${(\Lambda_\mathbb{Z}[w]\otimes\mathbb{Z}[c,v])}/{\langle c^{2}, 2c\cdot v, w\cdot c\rangle}$, that $\Delta(v^k)=0$, that $\Delta(c\cdot v^k)=0$, and that $\Delta(w\cdot v^k)=(2k+1)v^q+(k+1)c\cdot v^{k+1}$.  The result follows by taking $a=c$, $b=w$, and by replacing $v$ with $v+c\cdot v^2$.
\end{proof}

\begin{theorem*}[Yang \cite{\Yang}]
The rational string topology BV algebra of $\CP^n$ is given by
\[\mathbb{H}_\ast(L\CP^n;\mathbb{Q})=\frac{\mathbb{Q}[x,u,t]}{\langle x^{n+1},u^2,x^n\cdot t,u\cdot x^n\rangle}\]
with $|x|=-2$, $|u|=-1$ and $|t|=2n$, and $\Delta(t^k\cdot x^l)=0$, $\Delta(t^k\cdot u\cdot x^l)=[-(k+1)n-k+l]t^k\cdot x^l$.
\end{theorem*}
\begin{proof}
Since $\mathbb{H}_\ast(L\CP^n;\mathbb{Q})=\mathbb{H}_\ast(L\CP^n;\mathbb{Z})\otimes\mathbb{Q}$ as BV algebras, the two theorems of \S\ref{ResultSubsection} give us $\mathbb{H}_\ast(L\CP^n;\mathbb{Q})={\mathbb{Q}[c,v,w]}/{\langle c^{n+1},c^n\cdot v, w\cdot c^n,w^2\rangle}$, $\Delta(c^l\cdot v^k)=0$ and $\Delta(c^l\cdot w\cdot v^k)=[(n-l)+k(n+1)]c^l\cdot v^k$ for $k,l\geqslant 0$.  The result follows by setting $x=c$, $u=-w$ and $t=v$.
\end{proof}

Even though Theorem A extends these results of Menichi and Yang, our proof is rather different from the proofs given by those authors.  Both Menichi and Yang take as their starting point the ring isomorphism of Cohen and Jones \cite{\CohenJones}
\[\mathbb{H}_\ast(LM)\cong HH^\ast(S^\ast(M),S^\ast(M)),\]
where $M$ is simply connected and $S^\ast(M)$ denotes singular cochains. They then proceed by largely algebraic methods.  On the other hand our proof, as we shall see below, is entirely topological, taking place always at the level of spaces and their homology.

\subsection{Relationship with Hochschild cohomology}

One of the most important issues in string topology is to understand the relationship between the string topology of a manifold and the Hochschild cohomology of its singular cochains.  As mentioned above, for any simply connected $M$ there is an isomorphism of rings $\mathbb{H}_\ast(LM)\cong HH^\ast(S^\ast(M),S^\ast(M))$.  Both sides of this isomorphism are in fact BV algebras --- the left hand side is the string topology BV algebra, and for the right hand side the BV structure has been defined by Tradler~\cite{\Tradler}.  It is therefore natural to ask whether there is an isomorphism
\begin{equation}\label{ConjectureEquation}\mathbb{H}_\ast(LM)\cong HH^\ast(S^\ast(M),S^\ast(M))\end{equation}
of \emph{BV algebras}.  Felix and Thomas \cite{\FelixThomas} have shown that this is true when one takes coefficients in a field of characteristic zero.  On the other hand, Menichi \cite{\Menichi} has shown that this is not the case when $M=S^2$ and coefficients are taken in $\mathbb{Z}/2\mathbb{Z}$.  

We will answer the question of whether there is an isomorphism \eqref{ConjectureEquation} when $M$ is a complex projective space and cofficients are taken in $\mathbb{Z}$.  This is a matter of comparing the two BV algebras $\mathbb{H}_\ast(L\CP^n)$ and $HH^\ast(S^\ast(\CP^n),S^\ast(\CP^n))$.  The description of the first was given in \S\ref{ResultSubsection}, and Yang has computed the second:
\begin{theorem*}[Yang \cite{\Yang}]
Take coefficients in $\mathbb{Z}$.  As a BV algebra,
\[HH^\ast(S^\ast(\CP^n),S^\ast(\CP^n))=\frac{R[x,u,t]}{\langle x^{n+1},u^2,(n+1)x^nt,ux^n\rangle},\]
where $|x|=-2$, $|u|=-1$ and $|t|=2n$, and
$\Delta(t^kx^l)=0$, $\Delta(t^k u x^l)=(-(k+1)n-k+l)t^kx^l$.
\end{theorem*}

\begin{corollary*}
Take coefficients in $\mathbb{Z}$.  When $n$ is even there is an isomorphism of BV algebras
\begin{equation}\label{SupposedIsomorphism}\mathbb{H}_\ast(L\CP^n)\xrightarrow{\cong} HH^\ast(S^\ast(\CP^n),S^\ast(\CP^n)),\end{equation}
and when $n$ is odd no such isomorphism exists.
\end{corollary*}
\begin{proof}
A ring isomorphism \eqref{SupposedIsomorphism} is given by sending  $c$, $w$ and $v$ to $x$, $-u$ and $t$ respectively.  Using Theorem A and Yang's result above it is simple to verify that, when $n$ is even, this is an isomorphism of BV algebras.

Now take $n$ to be odd and suppose for a contradiction that an isomorphism \eqref{SupposedIsomorphism} exists.   It must send $w$ to $\pm u$ and $c$ to $\pm x$.  The unit element is sent to the unit element.  Since $\Delta(w)=n1 + \binom{n+1}{2}c^n\cdot v$ by Theorem A, we can apply the isomorphism to find that $\Delta(u)=\pm n 1 \pm\binom{n+1}{2} x^n t$.  But $\Delta(u)=-n1$ by Yang's result above, and $1$ and $x^n t$ generate distinct summands in degree $0$, so we must have $\binom{n+1}{2} x^n t=0$.  This is not the case, and consequently no isomorphism \eqref{SupposedIsomorphism} exists.
\end{proof}

\subsection{Outline of the paper}

The paper begins by establishing several new general results that we hope will be of independent interest.  The results will be applied later as part of the proof of Theorem A. They are as follows:
\begin{enumerate}
\item \emph{String topology and Lie group actions.}
When a Lie group $G$ acts on a manifold $M$, one obtains actions of both $G$ and $\Omega G$ on $LM$. Thus $\mathbb{H}_\ast(LM)$ becomes a module over the two rings $H_\ast(G)$ and $H_\ast(\Omega G)$.  We study this situation in \S\ref{GroupActionsSection} and prove a new result, Theorem~\ref{ActionTheorem}, which describes how the two module structures interact with the BV algebra structure.  

\item \emph{Compactifications of vector bundles.}
In \S\ref{CompactificationsSection} we introduce the \emph{pinched compactification} of a complex vector bundle $\xi$.  This is obtained by adding a `point at infinity' in each fibre, and then identifying it with the origin in that fibre.  We relate the pinched compactification to the more familiar projective and spherical compactifications, and explain how the homology of the pinched compactification of $\xi$ can be understood in terms of the unit sphere bundle $\mathbb{S}\xi$.

\item \emph{Circle actions on sphere bundles.}
Given complex vector bundles $\xi,\eta$ over a space $X$, one can form the unit sphere bundle $\mathbb{S}(\xi\oplus\eta)$.  The circle acts on this space by scalar multiplication in the first summand.  We study this situation in \S\ref{SphereBundlesSection} and prove Theorem~\ref{SphereBundlesTheorem}, which uses the Chern classes of $\xi$ and $\eta$ to describe the associated degree-raising operator $H_\ast(\mathbb{S}(\xi\oplus\eta))\to H_{\ast+1}(\mathbb{S}(\xi\oplus\eta))$.  
\end{enumerate}
After establishing these general results we are able to begin the proof of Theorem A.  The proof is arranged as a series of successive reformulations --- Theorems B,C and D --- of Theorem A.  Broadly speaking, the steps are as follows:
\begin{enumerate}
\item \emph{The action of $U(n+1)$ on $\CP^n$.}
 In \S\ref{FirstReformulationSection} we consider the natural action of $U(n+1)$ on $\CP^n$.  We state Theorem B, which describes how $H_\ast(U(n+1))$, $H_\ast(\Omega U(n+1))$ and $\Delta$ act on the subring of $\mathbb{H}_\ast(L\CP^n)$ generated by $c$ and $w$.  We then apply the results of \S\ref{GroupActionsSection} to show that Theorem A follows from Theorem B.

\item
\emph{A finite dimensional approximation of $L\CP^n$.}
In \S\ref{ApproximationSection} we construct a finite dimensional space $\mathbb{L}^n$ and a map $\mathbb{L}^n\to L\CP^n$ that induces an injection $H_\ast(\mathbb{L}^n)\to H_\ast(L\CP^n)$.  We can thus regard $\mathbb{L}^n$ as a finite dimensional approximation to $L\CP^n$.  Then in \S\ref{SecondReformulationSection} we state Theorem C, which describes the homology of $\mathbb{L}^n$, and use it to prove Theorem B.  The crucial point is that the homology of $\mathbb{L}^n$ is large enough that each of $c$, $w$ and $v$ must lie within its image.

\item
\emph{Relating $\mathbb{L}^n$ to $\mathbb{S}T\CP^n$.}
In \S\ref{LnIsACompactificationSection} we show that $\mathbb{L}^n$ is the pinched compactification of the tangent bundle $T\CP^n$.  We apply the general results of \S\ref{CompactificationsSection} to describe the homology of $\mathbb{L}^n$ in terms of $\mathbb{S}T\CP^n$.  Then in \S\ref{ThirdReformulationSection} we state Theorem D, which describes the homology of $\mathbb{S}T\CP^n$, and use it to prove Theorem C.

\item
\emph{Computing $H_\ast(\mathbb{S}T\CP^n)$.}
It remains to prove Theorem D, which describes $H_\ast(\mathbb{S}T\CP^n)$ and various module structures on it.  We do this by rephrasing the problem in terms of the situation considered in \S\ref{SphereBundlesSection} and then applying Theorem~\ref{SphereBundlesTheorem} several times.  The connection between $\Delta$ and characteristic classes of $\CP^n$ arises at this point.  This completes the proof of Theorem A.
\end{enumerate}

It is interesting to ask whether our methods can be applied to other manifolds besides $\CP^n$.  This is certainly the case for Lie groups, since Theorem~\ref{ActionTheorem} directly generalises the main result of  \cite{\StringLieGroups}.  We also believe that the string topology BV algebra for spheres and for real and quaternionic projective spaces could be computed in the same way; for spheres the result is known \cite{\Menichi}, but only partial information is available for the other projective spaces.

\section{String topology and Lie group actions}\label{GroupActionsSection}

In this section we will see how to exploit symmetries of a manifold when trying to compute its string topology BV algebra.  Applying this to the action of  $U(n+1)$ on $\CP^n$, we will be able to reformulate Theorem A in terms of slightly more accessible information.  Throughout the section homology groups are taken with coefficients in any commutative ring $R$.

Suppose that a closed, oriented manifold $M$ is equipped with an orientation preserving left action of a compact, connected Lie group $G$.  Then $LM$ inherits a left action of $LG$, so that $\mathbb{H}_\ast(LM)$ becomes a module over $H_\ast(LG)$.  It is natural to ask how this module structure interacts with the BV algebra structure described in the last paragraph.

It is convenient to write $LG$ as the semidirect product $\Omega G\rtimes G$, where the first factor is the subgroup of loops based at the identity and the second factor is the subgroup of constant loops.  Then $\mathbb{H}_\ast(LM)$ becomes a module both $H_\ast(\Omega G)$ and $H_\ast (G)$. We will describe how the BV algebra structure on $\mathbb{H}_\ast(LM)$ interacts with these two module structures, and to do so we must recall two quantities:
\begin{enumerate}
\item \emph{The Hopf algebra $H_\ast(\Omega G)$.}
For any pointed space $X$, the homology groups $H_\ast(\Omega X)$ form a graded ring under the Pontrjagin product.  When $X=G$, a theorem of Bott \cite{\EvenBott} states that these homology groups are free and concentrated in even degrees.  $H_\ast(\Omega G)$ then becomes a commutative, cocommutative Hopf algebra with coproduct
\[H_\ast(\Omega G)\xrightarrow{D_\ast} H_\ast(\Omega G\times\Omega G)\cong H_\ast(\Omega G)\otimes H_\ast(\Omega G)\]
obtained using the diagonal map and the K\"unneth isomorphism.  We denote the coproduct of $a\in H_\ast(\Omega G)$ by $D_\ast a=\sum a_{1}\otimes a_{2}$.

\item \emph{The homology suspension $\sigma\colon H_\ast(\Omega G)\to H_{\ast+1}(G)$.}
Let $X$ be a pointed space and let $\sigma\colon S^1\times\Omega X\to X$ denote the evaluation map.  The homology suspension, which we also denote by $\sigma$, is the homomorphism $H_\ast(\Omega X)\to H_{\ast+1}(X)$ defined by $\sigma(a)=\sigma_\ast([S^1]\times a)$ for $a\in H_\ast(\Omega X)$.  
\end{enumerate}
With this notation established we can state the main result of this section.

\begin{theorem}\label{ActionTheorem}
Let $x,y\in\mathbb{H}_\ast(LM)$.
\begin{enumerate}
\item For $a\in H_\ast(\Omega G)$ we have
\begin{equation}\label{OmegaProductEquation} a(x\cdot y)=(ax)\cdot y = x\cdot (ay)\end{equation}
and
\begin{equation}\label{OmegaDeltaEquation}\Delta(ax) = a\Delta(x) + \sum a_{1} \sigma(a_{2}) x.\end{equation}
\item For $b\in H_\ast(G)$ we have
\begin{equation}\label{GProductEquation}b(x\cdot y) = \sum (-1)^{|b_{2}| |x|}(b_{1}x)\cdot (b_{2}y)\end{equation}
and
\begin{equation}\label{GDeltaEquation}\Delta(bx) = (-1)^{|b|} b\Delta(x),\end{equation}
where in the first equation we have assumed that the image of $b$ under the diagonal $D_\ast\colon H_\ast(G)\to H_\ast(G\times G)$ can be written $\sum b_{1}\times b_{2}$.
\end{enumerate}
\end{theorem}

The crucial aspect of Theorem~\ref{ActionTheorem} is the following.
If $\Delta$ were a derivation of the loop product and we knew a set of generators for the ring $\mathbb{H}_\ast(LM)$, then $\Delta$ could be computed from its values on the generators.  In general $\Delta$ is not a derivation and only satisfies the weaker condition \eqref{BVIdentityEquation}.  This means that we must also compute its value on the product of each \emph{pair} of generators of $\mathbb{H}_\ast(LM)$.  Therefore the usefulness of formula \eqref{OmegaDeltaEquation}
\[\Delta((a1)\cdot x)=\Delta(ax)=\sum a_1\sigma(a_2)x + a\Delta(x),\]
is that it computes the value of $\Delta$ on certain products of elements of $\mathbb{H}_\ast(LM)$.

By taking $M=G$ under the natural left action, one can use Theorem~\ref{ActionTheorem} to give a complete account of the BV algebra $\mathbb{H}_\ast(LG)$.  This is the subject of \cite{\StringLieGroups}.  Another simple corollary identifies a particularly nice subring of $\mathbb{H}_\ast(LM)$:

\begin{corollary}
Let $1\in \mathbb{H}_0(LM)$ denote the unit.  Then the map
\begin{gather*}
H_\ast(\Omega G)\to\mathbb{H}_\ast(LM),\\
a\mapsto a1
\end{gather*}
is a ring homomorphism and $\Delta$ vanishes on its image.
\end{corollary}
\begin{proof}
Let $a,a'\in H_\ast(\Omega G)$.  Then by applying equation \eqref{OmegaProductEquation} we have $(a1)\cdot(a'1)=a(1\cdot(a'1))=a(a'1)=(aa')1$ so that $a\mapsto a1$ is indeed a ring homomorphism.  By equation \eqref{OmegaDeltaEquation}, $\Delta(a1)=\sum a_{1}\sigma(a_{2})1$.  Note that each $\sigma(a_{2})1$ has positive degree.  However $\sigma(a_{2})1$ lies in the subring $\mathbb{H}_\ast (M):=H_{\ast+m}(M)$ of $\mathbb{H}_\ast(LM)$ obtained from the inclusion of constant loops $M\hookrightarrow LM$.  Since $\mathbb{H}_\ast(M)$ is concentrated in non-positive degrees each $\sigma(a_{2})1$ must vanish.  Thus $\Delta(a1)=0$.
\end{proof}

The rest of this section is given to the proof of Theorem~\ref{ActionTheorem}.  In \S\ref{PontrjaginThomSubsection} we recall the Pontrjagin-Thom construction and several of its properties.  In \S\ref{LoopProductSubsection} we recall Cohen and Jones' definition of the loop product in terms of Pontrjagin-Thom maps and prove equations \eqref{OmegaProductEquation} and \eqref{GProductEquation}.  Finally, in \S\ref{DeltaSubsection} we recall the definition of the BV operator $\Delta$ and prove equations~\eqref{OmegaDeltaEquation} and \eqref{GDeltaEquation}.

\subsection{Pontrjagin-Thom maps}\label{PontrjaginThomSubsection}
Our proof of equations \eqref{OmegaProductEquation} and \eqref{GProductEquation} will use the homotopy theoretic description of the loop product due to Cohen and Jones \cite{\CohenJones}.  Their construction makes use of Pontrjagin-Thom maps, which we now recall.

Cohen and Klein \cite{\CohenKlein} explain that to a cartesian diagram
\begin{equation}\label{PTDiagram}\xymatrix{
E_1\ar[r]^{\tilde f}\ar[d]_{\pi_1} & E_2\ar[d]^{\pi_2}\\
P\ar@{^{(}->}[r]_f & N
}\end{equation}
in which $f$ is an embedding of smooth closed manifolds and $E_2\to N$ is a fibre bundle (though $E_2$ need not be a manifold), there is an induced Pontrjagin-Thom map
\[\tilde f_!\colon E_2\to E_1^{\nu_f}\]
that is well defined up to homotopy.  Here $\nu_f\to P$ denotes the normal bundle of $f$ and also its pullback to $E_1$.  The Pontrjagin-Thom map $\tilde f_!$ is obtained in the following way.  Identify the total space of $\nu_f\to P$ as a tubular neighbourhood of $f(P)$ in $N$, and correspondingly identify the total space of $\nu_f\to E_1$ as a neighbourhood of $\tilde f E_1$ inside $E_2$.  Then by collapsing the complement of this neighbourhood to a point we obtain the required map $\tilde f_!\colon E_2\to E_1^{\nu_f}$.  If $\nu_f$ is oriented, for example if $P$ and $N$ were oriented, then we obtain an umkehr map in homology, also denoted by $\tilde f_!$,
\[\tilde f_!\colon H_\ast(E_2)\to H_{\ast-d}(E_1)\]
as the composite $H_\ast(E_2)\to \tilde H_\ast(E_1^{\nu_f})\cong H_{\ast-d}(E_1)$ of the map induced by $\tilde f_!\colon E_2\to E_1^{\nu_f}$ with the Thom isomorphism.  Here $d=\dim\nu_f$.

In proving \eqref{OmegaProductEquation} and \eqref{GProductEquation} we will need to use several properties of the Pontrjagin-Thom construction.  We list these now.  First, we require the following naturality property.  Suppose we are given a second cartesian diagram
\[\xymatrix{
F_1 \ar[r]^{\tilde f}\ar[d]_{\rho_1} & F_2 \ar[d]^{\rho_2}\\
P \ar[r]_f & N
}\]
with $\rho_2$ a fibre bundle.  Suppose further that we have maps $g_1\colon F_1\to E_1$ and $g_2\colon F_2\to E_2$ such that $\pi_1\circ g_1=\rho_1$, $\pi_2\circ g_2 = \rho_2$.  Under $g_1$ the vector bundle $\nu_f\to E_1$ is pulled back to $\nu_f\to F_1$ and so there is an induced map of Thom spaces $g_1^{\nu_f}\colon F_1^{\nu_f}\to E_1^{\nu_f}$.  Then the diagram
\[\xymatrix{
F_2 \ar[r]^{\tilde f_!}\ar[d]_{g_2} & F_1^{\nu_f}\ar[d]^{g_1^{\nu_f}}\\
E_2 \ar[r]_{\tilde f_!}& E_1^{\nu_f}
}\]
commutes. If $\nu_f$ is oriented then we have the corresponding result
\begin{equation}\label{NaturalityEquation}
{g_1}_\ast\circ\tilde f_! = \tilde f_!\circ {g_2}_\ast
\end{equation}
for maps of homology groups.

Second, we require a factorisation property.  Suppose that diagram \eqref{PTDiagram} can be factorised into two cartesian squares
\begin{equation}\label{FactorizationDiagram}\xymatrix{
E_1 \ar[r]^{\tilde f =\tilde f'}\ar[d]_{\sigma_1} & E_2\ar[d]^{\sigma_2}\\
P'\ar@{^{(}->}[r]^{f'}\ar[d]_{} & N'\ar[d]^{}\\
P\ar@{^{(}->}[r]_f & N
}\end{equation}
in which $N'\to N$ is a fibre bundle of manifolds and $\sigma_2$ is any other fibre bundle.  We therefore obtain two Pontrjagin-Thom maps
\begin{eqnarray*}
\tilde f_!\colon E_2&\to& E_1^{\nu_f},\\
\tilde f'_!\colon E_2&\to& E_1^{\nu_{f'}},
\end{eqnarray*}
the first from \eqref{PTDiagram} and the second from the upper square of \eqref{FactorizationDiagram}.  However $\nu_{f'}\to P'$ can be identified with the pullback of $\nu_f\to P$, and so $\nu_f\to E_1$ and $\nu_{f'}\to E_1$ coincide.  Consequently $E_1^{\nu_f}=E_1^{\nu_{f'}}$.  Also, the tubular neighbourhood of $f'(P')$ in $N'$ can be chosen as the preimage of the tubular neighbourhood of $f(P)$ in $N$.  It follows that $\tilde f_!=\tilde f'_!$.  So long as the orientations of $\nu_f$ and $\nu_{f'}$ are chosen in a compatible way, the same formula
\begin{equation}\label{FactorizationEquation}
\tilde f'_!=\tilde f_!
\end{equation}
holds for the maps of homology groups.

Third, let $K$ be any other space and modify diagram \eqref{PTDiagram} to obtain a new cartesian square
\[\xymatrix{
K\times E_1 \ar[r]^{1\times\tilde f}\ar[d]_{\tau_1} &K\times E_2\ar[d]^{\tau_2}\\
P\ar@{^{(}->}[r]_f & N
}\]
in which $\tau_i\colon K\times E_i\to N$, defined by $\tau_i(k,e_i)=\pi_i(e_i)$, is a fibre bundle.  Then $\nu_f\to K\times E_2$ is the pullback of $\nu_f\to E_2$ and so $(K\times E_1)^{\nu_f}=K_+\wedge E_1^{\nu_f}$.  Then the Pontrjagin-Thom map $(1\times\tilde f)_!\colon K\times E_2\to K_+\wedge E_1^{\nu_f}$ can be identified as the composite $K\times E_2\to K\times E_1^{\nu_f}\to K_+\wedge E_1^{\nu_f}$ of $1\times \tilde f_!$ with the collapse map $K\times E_2^{\nu_f}\to K_+\wedge E_1^{\nu_f}$.  It follows that if $\nu_f$ is oriented, then in homology we have
\begin{equation}\label{ProductEquation}
(1\times\tilde f)_!(k\times e_2)=k\times\tilde f_!e_2
\end{equation}
for $k\in H_\ast(K)$ and $e_2\in H_\ast(E_2)$.

\subsection{The loop product}\label{LoopProductSubsection}
In this subsection we recall the definition of the loop product and prove equations \eqref{OmegaProductEquation} and \eqref{GProductEquation} from Theorem~\ref{ActionTheorem}.  By reversing the degree shift, we may regard the loop product $\mathbb{H}_\ast(LM)\otimes\mathbb{H}_\ast(LM)\to \mathbb{H}_\ast(LM)$ as a map $H_\ast(LM)\otimes H_\ast(LM)\to H_{\ast-m}(LM)$.
Cohen and Jones \cite{\CohenJones} identify this map in terms of the Pontrjagin-Thom construction, as follows.  The evaluation map $\ev\colon LM\to M$ that sends a loop in $M$ to its value at the basepoint $0\in S^1$ is a fibre bundle.  We write $L^2M$ for the space of pairs of composable loops in $M$.  In other words $L^2M\subset LM\times LM$ is the subset consisting of pairs $(\delta_1,\delta_2)$ with $\ev(\delta_1)=\ev(\delta_2)$.  Sending a pair of composable loops to their common basepoint defines a map $\ev\colon L^2M\to M$, and composing the loops defines a map $\gamma\colon L^2M\to LM$.  If we write $\tilde D\colon L^2M\to LM\times LM$ for the inclusion then we have a cartesian diagram
\[\xymatrix{
L^2M \ar[r]^-{\tilde D}\ar[d]_{\ev}& LM\times LM\ar[d]^{\ev\times\ev}\\
M\ar[r]_-{D} & M\times M
}\]
in which the lower map is an embedding of manifolds with normal bundle $TM$.   
We therefore have a Pontrjagin-Thom map
\[\tilde D_!\colon H_\ast(LM\times LM)\to H_{\ast-m}(L^2M).\]
as in the last section.  Cohen and Jones identify the loop product of $x,y\in\mathbb{H}_\ast(LM)=H_{\ast +m}(LM)$ as
\begin{equation}\label{LoopEquation}x\cdot y = (-1)^{m|y|+m}\gamma_\ast\circ \tilde D_!(x\times y).\end{equation}
Note that here $|y|$ denotes the degree of $y$ as an element of $\mathbb{H}_\ast(LM)$.

\begin{proof}[Proof of Theorem~\ref{ActionTheorem}, equation \eqref{GProductEquation}.]
Write $\alpha\colon G\times M\to M$ for the action of $G$ on $M$.  Consider the following commutative diagram.
\[
\xymatrix{ 
& G\times L^2M\ar[rr]\ar'[d][dd] & &G\times LM\times LM \ar[dd]
\\ 
L^2M \ar@{<-}[ur]^{L^2\alpha}\ar[rr]\ar[dd]& & LM\times LM \ar@{<-}[ur]^{L\alpha^2}\ar[dd] 
\\ 
& G\times M \ar'[r][rr] & & G\times M\times M
\\ 
M\ar[rr]\ar@{<-}[ur]_{\alpha} & & M\times M \ar@{<-}[ur]_{\alpha^2} 
}  
\]
Here the horizontal maps are given by $D\colon M\to M\times M$, $\tilde D\colon L^2M\to LM\times LM$, and their products with $1\colon G\to G$.  The vertical maps are given by $\ev\colon L^2M\to M$, $\ev\times\ev \colon LM\times LM\to M\times M$, and products of these with $1\colon G\to G$.  If we set $L\alpha(g,\delta)(t)=g\delta(t)$ then the various diagonal maps are defined by $\alpha^2(g,m_1,m_2)=(gm_1,gm_2)$, $L^2\alpha(g,(\delta_1,\delta_2))=(L\alpha(g,\delta_1),L\alpha(g,\delta_2))$, $L\alpha^2(g,(\delta_1,\delta_2))=(L\alpha(g,\delta_1),L\alpha(g,\delta_2))$ for $g\in G$, $m_1,m_2\in M$, $\delta,\delta_1,\delta_2\in LM$ and $t\in S^1$.  In particular we have
\[bx = L\alpha_\ast(b\times x).\]
for any $b\in H_\ast(G)$ and $x\in\mathbb{H}_\ast(LM)$.

To begin we claim that
\[\sum (-1)^{|b_{2}||x|}(b_{1}x)\cdot (b_{2}y)=(-1)^{m|y|+m}\gamma_\ast\circ\tilde D_!\circ L\alpha^2_\ast(b\times x\times y).\]
To see this, apply the definition of the loop product to the left hand side above and use the fact that $\sum (-1)^{|b_{2}|(|x|+m)}(b_{1}x)\times (b_{2}y)=L\alpha^2_\ast(b\times x\times y)$ (recall that $|x|$ refers to the degree of $x$ as an element of $\mathbb{H}_\ast(LM)$).  By naturality of umkehr maps \eqref{NaturalityEquation} in the above diagram we have
\[\tilde D_!\circ L\alpha^2_\ast(b\times x\times y)=L^2\alpha_\ast\circ (1\times\tilde D)_!(b\times x\times y).\]
By the factorisation formula \eqref{FactorizationEquation} it follows that $(1\times \tilde D)_!$ can be computed using just the rear square of the above diagram, so that
\[(1\times\tilde D)_!(b\times x\times y)= b\times\tilde D_!(x\times y).\]
It is easily seen that $\gamma \circ L^2\alpha= L\alpha\circ(1\times\gamma)$, so that 
\[\gamma_\ast \circ L^2\alpha_\ast = L\alpha_\ast\circ(1\times\gamma_\ast).\]
By combining the four equations obtained in this paragraph we find that
\begin{eqnarray*}
\sum(-1)^{|b_{2}| |x|}(b_{1}x)\cdot(b_{2}y)
&=&(-1)^{m|y|+m}L\alpha_\ast(b\times \gamma_\ast\circ \tilde D_!(x\times y))\\
&=&b(x\cdot y)
\end{eqnarray*}
as required.
\end{proof}

\begin{definition}
Let $\mu\colon\Omega G\times LM\to LM$ be defined by
\[\mu(\varepsilon,\delta)(t)=\left\{\begin{array}{lc}\varepsilon(2t)\delta(0), & 0\leqslant t\leqslant 1/2 \\ \delta(2t-1),& 1/2\leqslant t\leqslant 1       \end{array}\right.\]
for $\varepsilon\in\Omega G$, $\delta\in LM$, $t\in S^1$.  This definition makes sense since $\varepsilon(0)=\varepsilon(1)$ is the identity of $G$.
\end{definition}

\begin{lemma}\label{MuLemma}
The action of $H_\ast(\Omega G)$ on $\mathbb{H}_\ast(LM)$ can be described using $\mu$.  In other words, for $a\in H_\ast(\Omega G)$ and $x\in \mathbb{H}_\ast(LM)$ we have $ax = \mu_\ast(a\times x)$.
\end{lemma}
\begin{proof}
Consider the homotopy $h\colon[0,1]\times\Omega G\times LM\to LM$ defined by $h(s,\varepsilon,\delta)(t)=E(\varepsilon,s,t)D(\varepsilon,s,t)$ for $s\in[0,1]$, $\varepsilon\in\Omega G$, $\delta\in LM$, $t\in S^1$.  Here
\[E(\varepsilon,s,t)=\left\{\begin{array}{lc}\varepsilon(t/(1-s/2)), & 0\leqslant t\leqslant 1-s/2 \\ \varepsilon(1),& 1-s/2\leqslant t\leqslant 1       \end{array}\right.\]
and
\[D(\delta,s,t)=\left\{\begin{array}{lc} \delta(0), & 0\leqslant t\leqslant s/2 \\ \delta((t-s/2)/(1-s/2)),& s/2\leqslant t\leqslant 1.   \end{array}\right.\]
Then $h(0,\varepsilon,\delta)(t)=\varepsilon(t)\delta(t)$ and $h(1,\varepsilon,\delta)(t)=\mu(\varepsilon,\delta)(t)$, and the result follows.
\end{proof}

\begin{proof}[Proof of Theorem~\ref{ActionTheorem}, equation \eqref{OmegaProductEquation}.]
Let $\mu^2\colon\Omega G\times L^2M \to L^2M$ be defined by $\mu^2(\varepsilon,(\delta_1,\delta_2))=(\mu(\varepsilon,\delta_1),\delta_2)$.  Consider the diagram below.
\[
\xymatrix{ 
& \Omega G\times L^2M\ar[rr]^{1\times\tilde D}\ar'[d][dd] & &\Omega G\times LM\times LM \ar[dd]
\\ 
L^2M \ar@{<-}[ur]^{\mu^2}\ar[rr]^{\tilde D}\ar[dd]& & LM\times LM \ar@{<-}[ur]^{\mu\times 1}\ar[dd]
\\ 
& M \ar'[r][rr] & & M\times M
\\ 
M\ar[rr]\ar@{=}[ur] & & M\times M \ar@{=}[ur]
}  
\]
Here the horizontal maps are given by $D\colon M\to M\times M$, $\tilde D\colon L^2M\to LM\times LM$, and the product of $\tilde D$ with $1\colon \Omega G\to \Omega G$.  The vertical maps are given by $\ev\colon L^2M\to M$, $\ev\times\ev\colon LM\times LM\to M\times M$, and their compositions with the projections $\Omega G\times L^2M\to L^2M$, $\Omega G\times LM\times :M\to LM\times LM$.  This diagram is commutative, the vertical maps are fibre bundles, and the front and back faces are cartesian.

By Lemma~\ref{MuLemma} we have $(ax)\times y=(\mu\times 1)_\ast(a\times x\times y)$, and so by definition
\[(ax)\cdot y =  \gamma_\ast\circ\tilde D_!\circ(\mu\times 1)_\ast(a\times x\times y).\]
By the naturality property of umkehr maps \eqref{NaturalityEquation} we have
\[\tilde D_!(\mu\times 1)_\ast(a\times x\times y)=\mu^2_\ast\circ(1\times\tilde D)_!(a\times x\times y)\]
and by the product formula \eqref{ProductEquation} we have
\[(1\times\tilde D)_!(a\times x\times y)=a\times\tilde D_!(x\times y).\]
It is easy to verify that $\gamma\circ\mu^2$ is homotopic to $\mu\circ(1\times\gamma)$, and so
\[\gamma_\ast\circ\mu^2_\ast = \mu_\ast \circ (1\times\gamma)_\ast.\]
Combining the last four equations we find that
\[(ax)\cdot y = \mu_\ast(a\times\gamma_\ast\circ\tilde D_!(x\times y))=a(x\cdot y)\]
as required, and the formula $x\cdot(ay)=a(x\cdot y)$ follows similarly.
\end{proof}

\subsection{The BV operator}\label{DeltaSubsection}
In this subsection we recall the definition of the BV operator $\Delta$ and prove equations \eqref{OmegaDeltaEquation} and \eqref{GDeltaEquation} from Theorem~\ref{ActionTheorem}.  Let $\rho\colon S^1\times LM\to LM$ be the rotation action defined by $\rho(s,\delta)(t)=\delta(s+t)$ for $\delta\in LM$ and $s,t\in S^1$.  Then $\Delta$ is defined by $\Delta(x) = \rho_\ast([S^1]\times x)$ for $x\in\mathbb{H}_\ast(LM)$.

\begin{proof}[Proof of Theorem~\ref{ActionTheorem}, equations \eqref{OmegaDeltaEquation} and \eqref{GDeltaEquation}.]
Let us denote the action of $G$ on $M$ by $\alpha\colon G\times M\to M$, and write $L\alpha\colon G\times LM\to LM$ for action described by $L\alpha(g,\delta)(t)=g\delta(t)$ for $g\in G$, $\delta\in LM$ and $t\in S^1$.  Thus $bx=L\alpha_\ast(b\times x)$ for $b\in H_\ast G$ and $x\in\mathbb{H}_\ast(LM)$.  We have a commutative diagram
\[\xymatrix{
S^1\times G\times LM\ar[d]_{1\times L\alpha}\ar[r]^{t} & G\times S^1\times LM\ar[r]^-{1\times\rho} & G\times LM\ar[d]^{L\alpha} \\
S^1\times LM\ar[rr]_{\rho} & & LM
}\]
in which $t$ transposes factors.  It immediately follows that $\Delta(bx)=(-1)^{|b|}b\Delta x$, which proves \eqref{GDeltaEquation}.

Let us write $\Omega\alpha\colon\Omega G\times LM\to LM$ for the action of $\Omega G$ on $LM$ defined by $\Omega\alpha(\varepsilon,\delta)(t)=\varepsilon(t)\delta(t)$ for $\varepsilon\in\Omega G$, $\delta\in LM$ and $t\in S^1$.  Then for $a\in H_\ast(\Omega G)$ and $x\in\mathbb{H}_\ast(LM)$ we have $ax=\Omega\alpha_\ast(a\times x)$.  Let $\tau\colon S^1\times\Omega G\to\Omega G$ be defined by $\tau(s,\varepsilon)(t)=\varepsilon(s+t)\varepsilon(s)^{-1}$ for $\varepsilon\in\Omega G$ and $s,t\in S^1$.  Recall that $\sigma\colon S^1\times\Omega G\to G$ is defined by $\sigma(t)(\varepsilon)=\varepsilon(t)$ for $\varepsilon\in\Omega G$ and $t\in S^1$.  Then we have a commutative diagram
\[\xymatrix{
S^1\times\Omega G\times LM\ar[d]_{D\times D\times 1}\ar[r]^{1\times\Omega\alpha} & S^1\times LM\ar[r]^\rho & LM \\
{\scriptstyle(S^1\times S^1 \times S^1)\times (\Omega G\times \Omega G)\times LM}\ar[r]_-{s} & {\scriptstyle(S^1\times\Omega G)\times(S^1\times\Omega G)\times (S^1\times LM)}\ar[r]_-{\tau\times\sigma\times\rho} & {\scriptstyle\Omega G \times G \times LM}\ar[u]_{\Omega\alpha\circ(1\times L\alpha)}
}\]
in which $s$ shuffles the factors.

Note that $\rho_\ast(1\times x)=x$, that $\tau_\ast(1\times a_1)=a_1$, and that $\sigma_\ast(1\times a_2)=0$ unless $|a_2|=0$.  Also $D_\ast[S^1]=([S^1]\times 1\times 1)+(1\times[S^1]\times 1)+(1\times 1\times [S^1])$.  Then the commutativity of the diagram above shows that
\[\Delta(ax)=\tau_\ast([S^1]\times a)x+(-1)^{|a_1|}\sum a_1\sigma(a_2)x + (-1)^{|a|}a\Delta x.\]
However, since $H_\ast(\Omega G)$ is concentrated in even degrees, we have that $(-1)^{|a_1|}=(-1)^{|a|}=1$ and $\tau_\ast([S^1]\times a)=0$.  Equation \eqref{OmegaDeltaEquation} follows.
\end{proof}

\section{Compactifications of vector bundles}\label{CompactificationsSection}

In this section we continue to establish general results that will be used later in the proof of Theorem A.  In \S\ref{CompactificationsDefinitionSubsection} we will describe three different fibrewise compactifications of a complex vector bundle $\xi\to X$, the third being the `pinched compactification' that we are most interested in.  Then in \S\ref{CompactificationsHomologySubsection} we will show how to compute the homology of the pinched compactification of $\xi$ in terms of the unit sphere bundle $\mathbb{S}\xi$.  We will see in later sections that the pinched compactification of $T\CP^n$ serves as a useful approximation to the free loop space $L\CP^n$.  Throughout the section homology groups are taken with coefficients in any commutative ring $R$.

\subsection{Three compactifications of complex vector bundles}\label{CompactificationsDefinitionSubsection}

Let $\xi\to X$ be a complex vector bundle over a space $X$.  We will consider three different fibrewise compactifications of the total space $\xi$ and the relationship between them.

\begin{definition}
Let $d$ denote the complex rank of $\xi$.
\begin{enumerate}
\item The \emph{projective compactification} $\xi^\mathbb{P}$ of $\xi$ is the projectivisation $\mathbb{P}(\xi\oplus\mathbb{C})$ of the vector bundle $\xi\oplus\mathbb{C}\to X$.  It is a bundle of projective spaces $\CP^{d}$ over $X$. 
\item The \emph{spherical compactification} $\xi^\mathrm{S}$ of $\xi$ is the spherical part $\mathbb{S}(\xi\oplus\mathbb{R})$ of the real vector bundle $\xi\oplus\mathbb{R}\to X$.  It is a bundle of spheres of dimension $2d$ over $X$.
\item The \emph{pinched compactification} $\xi^\mathrm{P}$ of $\xi$ is the space obtained from $\xi^\mathrm{S}$ by identifying, for each $x\in X$, the points $(0,\pm 1)$ in the fibre over $x$.  It is a fibre bundle over $X$ whose typical fibre is a sphere of dimension $2d$ with its north and south poles identified.
\end{enumerate}
We have referred to $\xi^\mathbb{P}$, $\xi^\mathrm{S}$ and $\xi^\mathrm{P}$ as \emph{compactifications} of $\xi$ because each admits a natural inclusion of $\xi$ as a dense open subset.  In each case the inclusion commutes with the projections to $X$ and each fibre of $\xi$ is a dense open subset of the fibre of the compactification over $x$.  The inclusions are defined as follows:
\begin{enumerate}
\item The inclusion $\xi\hookrightarrow\xi^\mathbb{P}$ sends a vector $v\in\xi_x$ to the line spanned by $(v,1)$ in $\xi_x\oplus\mathbb{C}=(\xi\oplus\mathbb{C})_x$.
\item The inclusion $\xi\hookrightarrow\xi^\mathrm{S}$ sends a vector $v\in\xi_x$ to the unit vector
\[\left(\frac{2v}{\|v\|^2+1},\frac{\|v\|^2-1}{\|v\|^2+1}\right)\]
in $\xi_x\oplus\mathbb{R}$.
\item The inclusion $\xi\hookrightarrow \xi^\mathrm{P}$ is obtained by composing the inclusion $\xi\hookrightarrow\xi^\mathrm{S}$ with the identification map $\xi^\mathrm{S}\to\xi^\mathrm{P}$.
\end{enumerate}
The projective compactification $\xi^\mathbb{P}$ is obtained by adding to each fibre $\xi_x$ a copy of $\mathbb{P}(\xi_x)$ `at infinity'.  Indeed, there is a natural inclusion $\mathbb{P}\xi\hookrightarrow\mathbb{P}(\xi\oplus\mathbb{C})=\xi^\mathbb{P}$.  The spherical compactification is obtained by adding a single `point at infinity' in each fibre.  It therefore admits two distinct sections: the usual zero section, as well as what one might call the `section at infinity'.  The pinched compactification is obtained from the spherical compactification by identifying these two sections pointwise.  One  does not have to add any points to $\xi$ in order to form the pinched compactification; rather, one specifies that in each fibre a sequence of vectors tending to infinity has its limit at the origin.
\end{definition}

\begin{proposition}
There are maps
\[ \xi^\mathbb{P} \xrightarrow{p}  \xi^\mathrm{S} \xrightarrow{q} \xi^\mathrm{P}\]
that commute with the projections to $X$ and with the inclusions from $\xi$.  The map $p$ induces a homeomorphism between $\xi^\mathrm{S}$ and the space obtained from $\xi^\mathbb{P}$ by identifying, for each $x\in X$, the subset $\mathbb{P}\xi_x\subset\mathbb{P}\xi\subset\xi^\mathbb{P}$ to a point.  The map $q$ is the identification map obtained from the definition of $\xi^\mathrm{P}$.  It identifies, for each $x\in X$, the points $(0,\pm 1)$ in the fibre over $x$.
\end{proposition}
\begin{proof}
The assertions about $q$ are immediate from the definition of $\xi^\mathrm{P}$.  To construct $p$ we will begin with a description of $\xi^\mathbb{P}$.  A point of $\xi^\mathbb{P}$ is a line in $\xi\oplus\mathbb{C}$, and so can be projected to both $\xi$ and to $\mathbb{C}$.  The image of this projection is either itself a line, or is zero, and it is not possible for both projections to vanish.  Thus $\xi^\mathbb{P}$ is the union of two open subsets.  The first is the set of lines whose projection to $\mathbb{C}$ is nonzero.  The second is the set of lines whose projection to $\xi$ is nonzero.

The first subset of $\xi^\mathbb{P}$ described above is simply the image of the open embedding $\xi\hookrightarrow\xi^\mathbb{P}$.  To describe the second we must consider the tautological line bundle $\tau_\xi$ over $\mathbb{P}\xi$.  The fibre of $\tau_\xi$ at the point represented by a line $l$ in some $\xi_x$ is simply the line $l$ itself.  There is an open embedding $\tau_\xi\hookrightarrow\xi^\mathbb{P}$ that sends a vector $v$ in a line $l$ in $\xi_x$ to the span of $(v,\|v\|^2)$ if $v\neq 0$ and to $l\oplus 0$ if $v=0$. The second subset described above is simply the image of this embedding.  The intersection of the two open subsets corresponds under this embedding to the complement of the zero section in $\tau_\xi$.

The map $p\colon\xi^\mathbb{P}\to\xi^\mathrm{S}$, if it exists, is determined by the requirement that it should commute with the inclusions from $\xi$.  It is easy to compute that its restriction to $\tau_\xi$ would then be a map $\tau_\xi\to\xi^\mathbb{S}$ that on the complement of the zero-section sends an element $v\in l\subset\xi_x$ to
\[\left(\frac{2v}{1+\|v\|^2},\frac{1-\|v\|^2}{1+\|v\|^2}\right)\]
in $\xi_x\oplus\mathbb{C}$.  But this clearly extends to a continuous map $\tau_\xi\to\xi^\mathbb{S}$ that induces a homeomorphism between its image and the space obtained from $\tau_\xi$ by identifying, for each $x\in X$, the subset $\mathbb{P}\xi_x$ of the zero section to a single point.  But, conversely, this is sufficient to guarantee the existence of $p$, to verify that it commutes with the inclusions from $\xi$ and the projections to $X$, and to guarantee that it induces the required homeomorphism.
\end{proof}

\subsection{The homology of $\xi^\mathrm{P}$}\label{CompactificationsHomologySubsection}

In this subsection we wish to discuss the homology groups of the compactification $\xi^\mathrm{P}$ of $\xi$.  There are well-known techniques for describing the homology of $\xi^\mathbb{P}$ and $\xi^\mathrm{S}$.  We are also interested in investigating the effect in homology of the collapse map $q\circ p\colon\xi^\mathbb{P}\to\xi^\mathrm{P}$.  The results will be phrased in terms of the spherical bundle $\mathbb{S}\xi\to X$, which consists of all unit vectors in $\xi$.  We begin by introducing some notation and then stating the main results.  The rest of the subsection is given to proving these results, as well as a short lemma which we state and prove at the end.

\begin{definition}\label{TauDefinition}
$\xi^\mathbb{P}$ admits a tautological line bundle $\tau_{\xi}$.  A point of $\xi^\mathbb{P}$ is a line $l$ in a fibre $\xi_x\oplus\mathbb{C}$, and the fibre of $\tau_\xi$ over this point is simply $l$ itself.  We write $u_{\xi}\in H^2(\xi^\mathbb{P})$ for the first Chern class of $\tau_\xi$.
\end{definition}

\begin{note}
Suppose that $X$ is an oriented manifold.  Then $\xi$ and $\xi^\mathbb{P}$, as bundles of complex manifolds over $X$, inherit natural orientations from $X$.  The unit sphere bundle $\mathbb{S}\xi$ also inherits a natural orientation as the boundary of the unit ball bundle $\mathbb{B}\xi$, which like $\xi$ and $\xi^\mathbb{P}$ inherits an orientation from $X$.  Note in particular that $\partial_\ast[\mathbb{B}\xi,\mathbb{S}\xi]=[\mathbb{S}\xi]$.
\end{note}

\begin{proposition}\label{PinchedHomologyProposition}
There is a split short exact sequence of abelian groups
\begin{equation}\label{ShortExactHomologySequence}0\to H_\ast(X)\to H_\ast(\xi^\mathrm{P})\to H_{\ast-1}(\mathbb{S}\xi)\to 0.\end{equation}
Moreover, since each of $X$, $\xi^\mathrm{P}$ and $\mathbb{S}\xi$ has a projection map to $X$, each of the groups in \eqref{ShortExactHomologySequence} is naturally a $H^\ast(X)$ module, and \eqref{ShortExactHomologySequence} is in fact a split short exact sequence of $H^\ast(X)$ modules.
\end{proposition}

\begin{proposition}\label{ProjectiveToPinchedProposition}
For any $x\in H_\ast(\xi^\mathbb{P})$ the class $u_{\xi}\cap x$ vanishes under the composition
\[ H_\ast(\xi^\mathbb{P})\to H_\ast(\xi^\mathrm{P})\to H_{\ast-1}(\mathbb{S}\xi).\]
Moreover, when $X$ is a closed oriented manifold, so that $\xi^\mathbb{P}$ and $\mathbb{S}\xi$ become closed oriented manifolds, the composition above sends $[\xi^\mathbb{P}]$ to $[\mathbb{S}\xi]$.
\end{proposition}

Proposition~\ref{PinchedHomologyProposition} gives a complete description of $H_\ast(\xi^\mathrm{P})$.  There is also a classic description of $H^\ast(\xi^\mathbb{P})$ in terms of the Chern classes of $\xi\oplus\mathbb{C}$ that, when $X$ is a closed oriented manifold, can be Poincar\'e dualised to give a complete description of $H_\ast(\xi^\mathbb{P})$.  Proposition~\ref{ProjectiveToPinchedProposition} can then be used to give an exact description of $(q\circ p)_\ast\colon H_\ast(\xi^\mathbb{P})\to H_\ast(\xi^\mathrm{P})$.  Although we have not formulated such a corollary here, we will still need the full strength of Propositions~\ref{PinchedHomologyProposition} and \ref{ProjectiveToPinchedProposition} in later sections.

\begin{proof}[Proof of Proposition~\ref{PinchedHomologyProposition}]
The projection $\xi^\mathrm{P}\to X$ is a left-inverse to the zero section $X\hookrightarrow\xi^\mathrm{P}$.  The long exact sequence in homology for $(\xi^\mathrm{P},X)$ therefore splits into a collection of split short exact sequences
\begin{equation}\label{ProofShortExactSequence}0\to H_\ast(X)\to H_\ast(\xi^\mathrm{P})\to H_\ast(\xi^\mathrm{P},X)\to 0.\end{equation}
We must identify $H_\ast(\xi^\mathrm{P},X)$.

Let $I$ denote $[\frac{1}{2},1]$.  There is an inclusion $I\times \mathbb{S}\xi\hookrightarrow \xi\hookrightarrow \xi^\mathrm{P}$ that sends $(r,v)$ to $rv$.  Let $A\subset\xi\subset \xi^\mathrm{P}$ denote the subset consisting of vectors with length in $(\frac{1}{2},1)$.  We then have homology equivalences of pairs of subsets of $\xi^\mathrm{P}$
\begin{equation}\label{HomologyEquivalencesEquation}
(\xi^\mathrm{P},X)\hookrightarrow(\xi^\mathrm{P},\xi^\mathrm{P}\setminus A)\hookleftarrow(I,\partial I)\times\mathbb{S}\xi.
\end{equation}
There is also an isomorphism $\Theta\colon H_\ast((I,\partial I)\times\mathbb{S}\xi)\to H_{\ast-1}(\mathbb{S}\xi)$ given by composing $\partial_\ast$ with the projection $H_{\ast-1}(\partial I\times\mathbb{S}\xi)\to H_{\ast-1}(\mathbb{S}\xi)$ onto the summand corresponding to $1\in\partial I$.  We therefore have an isomorphism $H_\ast(\xi^\mathrm{P},X)\cong H_{\ast-1}(\mathbb{S}\xi)$ obtained from the zig-zag
\begin{equation}\label{ZigZagEquation}H_\ast(\xi^\mathrm{P},X)\xrightarrow{\cong} H_\ast(\xi^\mathrm{P},\xi^\mathrm{P}\setminus A)\xleftarrow{\cong} H_\ast((I,\partial I)\times\mathbb{S}\xi)\xrightarrow{\Theta} H_{\ast-1}(\mathbb{S}\xi).\end{equation}
Applying this isomorphism to \eqref{ProofShortExactSequence} provides us with the required short exact sequence \eqref{ShortExactHomologySequence}.  It remains to show that this is split as a short exact sequence of $H^\ast(X)$ modules.  

The sequence \eqref{ProofShortExactSequence} is certainly split as a sequence of $H^\ast(X)$ modules, since it is precisely the projection $\xi^\mathrm{P}\to X$ that provides the splitting.  We must therefore show that the isomorphism $H_\ast(\xi^\mathrm{P},X)\cong H_{\ast-1}(\mathbb{S}\xi)$ is an isomorphism of $H^\ast(X)$ modules.  Each of the pairs in \eqref{HomologyEquivalencesEquation}, as a pair of subsets of $\xi^\mathrm{P}$, admits a projection to the pair $(X,X)$, and the inclusion maps commute with these projections.  Consequently the first two isomorphisms in \eqref{ZigZagEquation} are isomorphisms of $H^\ast(X)$ modules.  We must now show that $\Theta$ is an isomorphism of $H^\ast(X)$ modules.  If we orient $I$ so that $\partial_\ast[I,\partial I]=[1]-[\frac{1}{2}]$ then $\Theta^{-1}$ is the map that sends that sends $x\in H_{\ast-1}(\mathbb{S}\xi)$ to $[I,\partial I]\times x$.  Now note that for $(I,\partial I)\times\mathbb{S}\xi$ the projection to $X$ factors through the projection $(I,\partial I)\times\mathbb{S}\xi\to (\mathbb{S}\xi,\mathbb{S}\xi)$, so that for $x\in H_\ast((I,\partial I)\times \mathbb{S}\xi)$ and $\alpha\in H^\ast(X)$ we have $\alpha x = (1\times\pi^\ast\alpha)\cap x$, where $\pi\colon\mathbb{S}\xi\to X$ denotes the projection.  In particular, for $y\in H_\ast(\mathbb{S}\xi)$ we have
\[\alpha ([I,\partial I]\times y) = (1\times\pi^\ast\alpha)\cap([I,\partial I]\times y)=[I,\partial I]\times (\pi^\ast\alpha\cap y).\]
(For the sign convention relating cross and cap products see \cite[p.255]{\Spanier}.)  It follows that $\Theta^{-1}$ is an isomorphism of $H^\ast(X)$ modules, and so the same is true of $\Theta$.  This completes the proof.
\end{proof}

\begin{proof}[Proof of Proposition~\ref{ProjectiveToPinchedProposition}]
Taking the definition of the second map from the proof Proposition~\ref{PinchedHomologyProposition}, it follows that the composite is given by the zig-zag
\[H_\ast(\xi^\mathbb{P})\to H_\ast(\xi^\mathrm{P})\to H_\ast(\xi^\mathrm{P},X)\xrightarrow{\cong} H_\ast(\xi^\mathrm{P},\xi^\mathrm{P}\setminus A)\xleftarrow{\cong} H_\ast((I,\partial I)\times \mathbb{S}\xi)\xrightarrow{\Theta} H_{\ast-1}(\mathbb{S}\xi)\]
where $A\subset \xi\subset \xi^\mathrm{P}$ is the subset consisting of vectors whose length lies in $(\frac{1}{2},1)$.  We can equally well regard $A$ as a subset of $\xi^\mathbb{P}$, and so there is a commutative diagram
\[\xymatrix{
(\xi^\mathbb{P},X\sqcup\mathbb{P}\xi)\ar@{^(->}[r]\ar[d]&    (\xi^\mathbb{P},\xi^\mathbb{P}\setminus A)\ar[d]  & (I,\partial I)\times\mathbb{S}\xi  \ar@{_(->}[l]\ar@{=}[d]\\
(\xi^\mathrm{P},X)\ar@{^(->}[r]  &  (\xi^\mathrm{P},\xi^\mathrm{P}\setminus A)  & (I,\partial I)\times\mathbb{S}\xi
\ar@{_(->}[l] 
}\]
whose horizontal maps are homology equivalences.  The zig-zag above is therefore equal to
\begin{equation}\label{NewZigZagEquation}H_\ast(\xi^\mathbb{P})\to H_\ast(\xi^\mathbb{P},X\sqcup\mathbb{P}\xi)\xrightarrow{\cong} H_\ast(\xi^\mathbb{P},\xi^\mathbb{P}\setminus A)\xleftarrow{\cong} H_\ast((I,\partial I)\times \mathbb{S}\xi)\xrightarrow{\Theta} H_{\ast-1}(\mathbb{S}\xi)\end{equation}
With this new description of the composite we can prove the proposition.

In the light of the last paragraph, to prove the first claim it will suffice to show that 
\[ H_\ast(\xi^\mathbb{P})\xrightarrow{u_{\xi}\cap -}H_{\ast-2}(\xi^\mathbb{P})\to H_{\ast-2}(\xi^\mathbb{P},X\sqcup\mathbb{P}\xi)\]
vanishes.  Since the restriction of $\tau_{\xi}$ to $X=\mathbb{P}(\mathbb{C})$ is trivial, the class $u_{\xi}$ lifts to a class $u\in H^2(\xi^\mathbb{P},X)$.  It follows that the first map above factors through ${u\cap-}\colon H_{\ast}(\xi^\mathbb{P},X)\to H_{\ast-2}(\xi^\mathbb{P})$ so it will suffice to show that
\[ H_\ast(\xi^\mathbb{P},X)\xrightarrow{u\cap -}H_{\ast-2}(\xi^\mathbb{P})\to H_{\ast-2}(\xi^\mathbb{P},X\sqcup\mathbb{P}\xi)\]
vanishes.  This will follow immediately from the fact that $u\cap -$ factors through $H_{\ast-2}(\mathbb{P}\xi)\to H_{\ast-2}(\xi^\mathbb{P})$, which we will now prove.

There is a homology equivalence $\phi\colon(\mathbb{B}(\tau_\xi),\mathbb{S}(\tau_\xi))\to (\xi^\mathbb{P},X)$.  A point in $\mathbb{B}(\tau_\xi)$ is a triple $(v,l,x)$ where $v$ is a vector in a line $l$ in a fibre $\xi_x$.  If $v\neq 0$ then $\phi$ sends this point to the line in $(\xi\oplus\mathbb{C})_x$ spanned by $((\frac{1}{\|v\|}-1)v,\|v\|)$.  If $v=0$ then $\phi$ sends the point to $l\oplus 0$. The composition of $\phi$ with the zero section $\mathbb{P}\xi\hookrightarrow\mathbb{B}(\tau_\xi)$ is just the inclusion $\mathbb{P}\xi\hookrightarrow\xi^\mathbb{P}$.  We now have a commutative diagram
\[\xymatrix{
H_\ast(\mathbb{B}(\tau_\xi),\mathbb{S}(\tau_\xi))\ar[r]^-{\phi^\ast u\cap-}\ar[d]_{\phi_\ast}^{\cong}   & H_\ast(\mathbb{B}(\tau_\xi))\ar[d]_{{\phi|}_\ast} & H_\ast(\mathbb{P}\xi)\ar[l]_-{\cong}\ar@{=}[d]\\
H_\ast(\xi^\mathbb{P},X)\ar[r]_{u\cap -} & H_\ast(\xi^\mathbb{P}) & H_\ast(\mathbb{P}\xi)\ar[l]
}\]
so that $u\cap -$ factors through $H_\ast(\mathbb{P}\xi)\to H_\ast(\xi^\mathbb{P})$ as required.

To prove the second part we must show that the zig-zag \eqref{NewZigZagEquation} sends $[\xi^\mathbb{P}]$ to $[\mathbb{S}\xi]$.  It is easy to see that the first three maps taken together send $[\xi^\mathbb{P}]$ to $[I\times \mathbb{S}\xi,\partial I\times\mathbb{S}\xi]$, where $I\times\mathbb{S}\xi$ is oriented as a subset of $\xi^\mathbb{P}$.  Since $\Theta$ is the composition of $\partial_\ast$ with the projection $H_{\ast-1}(\partial I\times\mathbb{S}\xi)\to H_{\ast-1}(\mathbb{S}\xi)$ onto the summand corresponding to $1\in\partial I$, it will suffice to show that $\partial_\ast[I\times\mathbb{S}\xi,\partial I\times\mathbb{S}\xi]=[1]\times [\mathbb{S}\xi] - [\frac{1}{2}]\times[\mathbb{S}\xi]$.  

Recall that $I\times\mathbb{S}\xi$ is oriented as a submanifold of $\xi^\mathbb{P}$.  The inclusion $I\times\mathbb{S}\xi\hookrightarrow\xi^\mathbb{P}$ factors as $I\times\mathbb{S}\xi\to\mathbb{B}\xi\to \xi^\mathbb{P}$.  The first map here sends $(r,v)\in I\times\mathbb{S}\xi$ to $rv$, and the second map is obtained from the inclusion $\xi\subset\xi^\mathbb{P}$.  In particular, the second map is an embedding of almost complex manifolds of the same dimension, so that the orientations coincide.  It follows that the orientation that $I\times\mathbb{S}\xi$ inherits from $\xi^\mathbb{P}$ is the same as the one it inherits from $\mathbb{B}\xi$.  But $\mathbb{S}\xi$ is oriented as the boundary of $\mathbb{B}\xi$, so that $\partial_\ast[\mathbb{B}\xi,\mathbb{S}\xi]=[\mathbb{S}\xi]$, and correspondingly $\partial_\ast[I\times\mathbb{S}\xi,\partial I\times\mathbb{S}\xi]=[1]\times [\mathbb{S}\xi] - [\frac{1}{2}]\times[\mathbb{S}\xi]$.
\end{proof}

\begin{lemma}\label{PinchedHomologyGroupActionLemma}
We refer to the situation of Proposition~\ref{PinchedHomologyProposition}.
Suppose that a group $G$ acts on $\xi^\mathrm{P}$ and that this action restricts to an action on $\xi$ that preserves the length of vectors. Then $G$ acts on each of $X$, $\xi^\mathrm{P}$ and $\mathbb{S}\xi$.  Thus each of the groups in \eqref{ShortExactHomologySequence} is a $H_\ast(G)$ module, and \eqref{ShortExactHomologySequence} is a short exact sequence of $H_\ast(G)$ modules, so long as we twist the module structure on $H_{\ast-1}(\mathbb{S}\xi)$ so that $\alpha\in H_\ast(G)$ acts on $H_{\ast-1}(\mathbb{S}\xi)$ by $(-1)^{|\alpha|}$ times the morphism defined by the action of $G$ on $\mathbb{S}\xi$.  If the action of $G$ is such that the projection $\xi^\mathrm{P}\to X$ is equivariant, then \eqref{ShortExactHomologySequence} is split as $H_\ast(G)$ modules.
\end{lemma}
\begin{proof}
The proof is directly analogous to the proof of the final part of Proposition~\ref{PinchedHomologyProposition}.  Each of the pairs in \eqref{HomologyEquivalencesEquation} inherits a $G$-action, so that $H_\ast(\xi^\mathrm{P},X)\cong H_\ast((I,\partial I)\times\mathbb{S}\xi)$ as $H_\ast(G)$ modules.  That $\Theta$ is an isomorphism of $H_\ast(G)$ modules now follows immediately from the fact that, for $\alpha\in H_\ast(G)$ and $x\in H_{\ast-1}((I,\partial I)\times\mathbb{S}\xi)$, we have $\partial_\ast(\alpha\times x)=(-1)^{|\alpha|}\alpha\times\partial_\ast x$.  The splitting is provided by the projection $\xi^\mathrm{P}\to X$, so that if this is $G$-equivariant, then the splitting is a splitting of $H_\ast(G)$ modules.
\end{proof}

\section{Circle actions on sphere bundles}\label{SphereBundlesSection}
This section presents the last of the general results that we will be using later to prove Theorem A.  We will consider certain circle actions on unit sphere bundles, and show how to compute the corresponding degree-raising operators on homology groups.  By applying this to $\mathbb{S}T\CP^n$ and to several associated bundles we will gain useful information about the string topology BV algebra of $\CP^n$.  This is because, as we shall see, $H_\ast(\mathbb{S}T\CP^n)$ is in fact a summand of $\mathbb{H}_\ast(L\CP^n)$.  Throughout the section we consider homology with coefficients in any commutative ring $R$.

Let us consider complex vector bundles over a space $X$.  Given such a vector bundle $\xi\to X$ we will write $\pi\colon\mathbb{S}\xi\to X$ for the associated unit sphere bundle.   Recall that the homology of $\mathbb{S}\xi$ can be computed using the \emph{Gysin sequence}:
\[\cdots\to H_\ast(\mathbb{S}\xi)\xrightarrow{\pi_\ast} H_\ast(X)\xrightarrow{E_\xi\cap-}H_{\ast-2d}(X)\to H_{\ast -1}(\mathbb{S}\xi)\to\cdots\]
Here $d$ denotes the rank of $\xi$ and $E_\xi$ denotes the Euler class.

\begin{definition}\label{CircleActionDefinition}
Let $\xi$ and $\eta$ be complex vector bundles over a base-space $X$ and let $S^1$ act on $\mathbb{S}(\xi\oplus\eta)$ in the following way: $t\in S^1$ sends a pair $(v,w)\in\mathbb{S}(\xi\oplus\eta)$ to $(tv,w)$.   Here we are regarding $S^1$ as the unit complex numbers acting by scalar multiplication.  The associated \emph{degree raising operator} is the homomorphism
\[R\colon H_\ast(\mathbb{S}(\xi\oplus\eta))\to H_{\ast+1}(\mathbb{S}(\xi\oplus\eta))\]
that sends $x\in H_\ast(\mathbb{S}(\xi\oplus\eta))$ to $\rho_\ast([S^1]\times x)$, where $\rho\colon S^1\times \mathbb{S}(\xi\oplus\eta)\to\mathbb{S}(\xi\oplus\eta)$ denotes the action.
\end{definition}

\begin{theorem}\label{SphereBundlesTheorem}
In the situation of Definition~\ref{CircleActionDefinition}, the degree-raising operator is equal to the composite
\begin{equation}\label{CompositeEquation}H_\ast(\mathbb{S}(\xi\oplus\eta))\to H_{\ast}(X)\xrightarrow{c_{d-1}(\xi)c_e(\eta)\cap-}H_{\ast-2d-2e+2}(X)\to H_{\ast+1}(\mathbb{S}(\xi\oplus\eta))\end{equation}
in which $d=\mathrm{rank}(\xi)$, $e=\mathrm{rank}(\eta)$, and the unlabelled maps are taken from the Gysin sequence.
\end{theorem}

The proof of Theorem~\ref{SphereBundlesTheorem} will be given in \S\ref{CircleActionsSubsection}.   Then in \S\ref{GysinSubsection} we establish some basic properties of the Gysin sequence that will be used later.

\subsection{Proof of Theorem~\ref{SphereBundlesTheorem}}\label{CircleActionsSubsection}

\begin{proof}[Proof of Theorem~\ref{SphereBundlesTheorem}]
Let us begin by recalling how to construct the Gysin sequence.  First form the pair $(\mathbb{B}\xi,\mathbb{S}\xi)$, where $\mathbb{B}\xi$ denotes the unit ball bundle of $\xi$, and take the associated long exact sequence of homology groups.  Then apply the isomorphisms
\[H_\ast(\mathbb{B}\xi)\xrightarrow{\pi_\ast}H_\ast(X),\qquad H_\ast(\mathbb{B}\xi,\mathbb{S}\xi)\xrightarrow{\Th}H_{\ast-d}(X)\]
to transform this sequence into the Gysin sequence.

The crucial point in the proof of the theorem is that we will use the definition of $c_{d-1}(\xi)$ given in \cite{\MilnorAndStasheff}, which we briefly recall.  The pullback bundle $\pi^\ast\xi$ admits a tautological inclusion $a\colon\mathbb{C}\hookrightarrow\pi^\ast\xi$ that in the fibre over $v\in\mathbb{S}\xi$ sends $z\in\mathbb{C}$ to $zv\in\xi_{\pi(v)}=(\pi^\ast\xi)_v$.  If we write $\xi^\perp$ for the orthogonal complement of this summand, then we have a splitting $\mathbb{C}\oplus\xi^\perp = \pi^\ast\xi$.  The cohomology Gysin sequence shows that $\pi^\ast\colon H^{2d-2}(X)\to H^{2d-2}(\mathbb{S}\xi)$ is an isomorphism, and $c_{d-1}(\xi)$ is defined to be the class whose image is $c_{d-1}(\xi^\perp)$.  Note that $c_{d-1}(\xi^\perp)$ is simply the Euler class of $\xi^\perp$.

To prove the theorem we must show that the degree raising operator $R$ coincides with \eqref{CompositeEquation}.  The proof will successively rephrase this problem until we can apply the fact described in the last paragraph.

The action map $\rho\colon S^1\times\mathbb{S}(\xi\oplus\eta)\to\mathbb{S}(\xi\oplus\eta)$ extends to a map of pairs, which we also denote by $\rho$,
\[\rho\colon (\mathbb{B}^2,S^1)\times\mathbb{S}(\xi\oplus\eta)\to(\mathbb{B}(\xi\oplus\eta),\mathbb{S}(\xi\oplus\eta))\]
which sends $(z,(v,w))\in\mathbb{B}^2\times\mathbb{S}(\xi\oplus\eta)$ to $(zv,w)\in\mathbb{B}(\xi\oplus\eta)$.  Here we are regarding $\mathbb{B}^2$ as the unit disc in the complex numbers.  Using this extended $\rho$ we can define an operator
\[R_1\colon H_\ast(\mathbb{S}(\xi\oplus\eta))\to H_{\ast+2}(\mathbb{B}(\xi\oplus\eta),\mathbb{S}(\xi\oplus\eta))\]
that sends $x\in H_\ast(\mathbb{S}(\xi\oplus\eta))$ to $\rho_\ast([\mathbb{B}^2,S^1]\times x)$.  This new $R_1$ is related to $R$ by the formula $\partial_\ast\circ R_1=R$.  The final map in \eqref{CompositeEquation} is $\partial_\ast\circ\Th^{-1}$.  These two facts mean that to prove the theorem we must show that the diagram
\begin{equation}\label{DiagramI}\xymatrix{
H_\ast(\mathbb{S}(\xi\oplus\eta))\ar[rr]^-{R_1}\ar[d]_{\pi_\ast}   & &  H_{\ast+2}(\mathbb{B}(\xi\oplus\eta),\mathbb{S}(\xi\oplus\eta))\ar[d]^\Th\\
H_{\ast}(X)\ar[r]_{c_{e}(\eta)\cap -}  & H_{\ast-2e}(X)\ar[r]_{c_{d-1}(\xi)\cap -} & H_{\ast-2d-2e+2}(X)
}\end{equation}
commutes.

By choosing an appropriate norm on $\xi\oplus\eta$ we can make identifications $\mathbb{B}(\xi\oplus\eta)=\mathbb{B}\xi\times_X\mathbb{B}\eta$ and $\mathbb{S}(\xi\oplus\eta)=(\mathbb{B}\xi\times_X\mathbb{S}\eta)\cup (\mathbb{S}\xi\times_X\mathbb{B}\eta)$.  There is an inclusion map
\[i\colon \mathbb{S}(\xi\oplus\eta)\to (\mathbb{S}(\xi\oplus\eta),\mathbb{B}\xi\times_X\mathbb{S}\eta),\]
and a projection map
\[p\colon(\mathbb{S}(\xi\oplus\eta),\mathbb{B}\xi\times_X\mathbb{S}\eta)\to(\mathbb{B}\eta,\mathbb{S}\eta)\]
The map $R_1$ appearing in \eqref{DiagramI} factors through $i_\ast$.  To see this, note that the extended $\rho$ can in fact be regarded as a map of pairs
\[\rho\colon (\mathbb{B}^2,S^1)\times(\mathbb{S}(\xi\oplus\eta), \mathbb{B}\xi\times_X\mathbb{S}\eta)\to(\mathbb{B}(\xi\oplus\eta),\mathbb{S}(\xi\oplus\eta))\]
defining an operator
\[R_2\colon H_\ast(\mathbb{S}(\xi\oplus\eta),\mathbb{B}\xi\times_X\mathbb{S}\eta)\to H_{\ast+2}(\mathbb{B}(\xi\oplus\eta),\mathbb{S}(\xi\oplus\eta))\]
that is related to $R_1$ by the equation $R_2\circ i_\ast=R_1$.  The map $c_{e}(\eta)\cap \pi_\ast(-)$ appearing in \eqref{DiagramI} also factors through $i_\ast$.  To see this, note that it can be written as
\[H_\ast(\mathbb{S}(\xi\oplus\eta))\xrightarrow{\pi_\ast}H_\ast(X)\xrightarrow{z_\ast} H_{\ast}(\mathbb{B}\eta)\to H_{\ast}(\mathbb{B}\eta,\mathbb{S}\eta)\xrightarrow{\Th} H_{\ast-2e}(X)\]
where $z\colon X\to\mathbb{B}\eta$ is the zero section.  But $z\circ\pi$ is homotopic to the projection $p\colon \mathbb{S}(\xi\oplus\eta)\to\mathbb{B}\eta$, and the composite $\mathbb{S}(\xi\oplus\eta)\to\mathbb{B}\eta\to(\mathbb{B}\eta,\mathbb{S}\eta)$ is equal to $p\circ i$.  Thus $c_{e}(\eta)\cap \pi_\ast(-)$ is equal to
\[H_\ast(\mathbb{S}(\xi\oplus\eta))\xrightarrow{i_\ast}H_\ast(\mathbb{S}(\xi\oplus\eta),\mathbb{B}\xi\times_X\mathbb{S}\eta)\xrightarrow{p_\ast} H_{\ast}(\mathbb{B}\eta,\mathbb{S}\eta)\xrightarrow{\Th} H_{\ast-2e}(X)\]
as claimed.

The last paragraph described how the two initial maps in \eqref{DiagramI} factor through $i_\ast$.   To show that \eqref{DiagramI} commutes it will therefore suffice to show that the diagram
\[\xymatrix{
H_\ast(\mathbb{S}(\xi\oplus\eta),\mathbb{B}\xi\times_X\mathbb{S}\eta)\ar[d]_{p_\ast}  \ar[rr]^-{R_2}& & H_{\ast+2}(\mathbb{B}(\xi\oplus\eta),\mathbb{S}(\xi\oplus\eta))\ar[d]^\Th  \\
H_{\ast}(\mathbb{B}\eta,\mathbb{S}\eta)\ar[r]_\Th &  H_{\ast-2e}(X)\ar[r]_-{c_{d-1}(\xi)\cap -} &  H_{\ast-2d-2e+2}(X)
}\]
commutes.  Moreover, since there is an excision-map
\begin{multline*}
j\colon (\mathbb{B}(\pi^\ast\eta),\mathbb{S}(\pi^\ast\eta))=(\mathbb{S}\xi\times_X\mathbb{B}\eta,\mathbb{S}\xi\times_X\mathbb{S}\eta)\to  \\
\to ((\mathbb{B}\xi\times_X\mathbb{S}\eta)\cup (\mathbb{S}\xi\times_X\mathbb{B}\eta),\mathbb{B}\xi\times_X\mathbb{S}\eta)=(\mathbb{S}(\xi\oplus\eta),\mathbb{B}\xi\times_X\mathbb{S}\eta)
\end{multline*}
it will suffice to show that the diagram
\begin{equation}\label{DiagramII}\xymatrix{
H_\ast(\mathbb{B}(\pi^\ast\eta),\mathbb{S}(\pi^\ast\eta))\ar[d]_{p_\ast\circ j_\ast}  \ar[rr]^-{R_2\circ j_\ast}& & H_{\ast+2}(\mathbb{B}(\xi\oplus\eta),\mathbb{S}(\xi\oplus\eta))\ar[d]^\Th  \\
H_{\ast}(\mathbb{B}\eta,\mathbb{S}\eta)\ar[r]_\Th &  H_{\ast-2e}(X)\ar[r]_-{c_{d-1}(\xi)\cap -} &  H_{\ast-2d-2e+2}(X)
}\end{equation}
commutes.

The definition of $R_2$, which appears in \eqref{DiagramII}, involves the pair $(\mathbb{B}^2,S^1)\times(\mathbb{B}(\pi^\ast\eta),\mathbb{S}(\pi^\ast\eta))$.   We can identify this pair with $(\mathbb{B}(\mathbb{C}\oplus\pi^\ast\eta),\mathbb{S}(\mathbb{C}\oplus\pi^\ast\eta))$.  We then have a commutative diagram
\[\xymatrix{
H_\ast(\mathbb{B}(\pi^\ast\eta),\mathbb{S}(\pi^\ast\eta))\ar[r]^{R_2\circ j_\ast}\ar[d]_{[\mathbb{B}^2,S^1]\times -} & H_{\ast+2}(\mathbb{B}(\xi\oplus\eta),\mathbb{S}(\xi\oplus\eta))\\
H_{\ast+2}(\mathbb{B}(\mathbb{C}\oplus\pi^\ast\eta),\mathbb{S}(\mathbb{C}\oplus\pi^\ast\eta))\ar[ur]_{A_\ast} &
}\]
where $A=\rho\circ(\mathrm{Id}\times j)$.  The composite $p\circ j$ which appears in \eqref{DiagramII} is just the projection $(\mathbb{B}(\pi^\ast\eta),\mathbb{S}(\pi^\ast\eta))\to(\mathbb{B}\eta,\mathbb{S}\eta)$.  Thus $\Th\circ p_\ast\circ j_\ast$ is equal to
\[H_\ast(\mathbb{B}(\pi^\ast\eta),\mathbb{S}(\pi^\ast\eta))\xrightarrow{\Th} H_{\ast-2e}(\mathbb{S}\eta)\xrightarrow{\pi_\ast} H_{\ast-2e}(X)\]
and we have a commutative diagram
\[\xymatrix{
H_\ast(\mathbb{B}(\pi^\ast\eta),\mathbb{S}(\pi^\ast\eta))\ar[r]^\Th\ar[d]_{[\mathbb{B}^2,S^1]\times -} & H_{\ast-2e}(\mathbb{S}\xi)\\
H_{\ast+2}(\mathbb{B}(\mathbb{C}\oplus\pi^\ast\eta),\mathbb{S}(\mathbb{C}\oplus\pi^\ast\eta))\ar[ur]_\Th & 
}\]
so that $\Th\circ q_\ast\circ j_\ast=\pi_\ast\circ \Th\circ ([\mathbb{B}^2,S^1]\times-)$.  Thus both composites in \eqref{DiagramII} factor through $([\mathbb{B}^2,S^1]\times -)$.  To show that the diagram commutes it thus suffices to show that
\begin{equation}\label{DiagramIII}\xymatrix{
{\scriptstyle H_{\ast+2}(\mathbb{B}(\mathbb{C}\oplus\pi^\ast\eta),\mathbb{S}(\mathbb{C}\oplus\pi^\ast\eta))}  \ar[rr]^{A_\ast}\ar[d]_\Th & & {\scriptstyle H_{\ast+2}(\mathbb{B}(\xi\oplus\eta),\mathbb{S}(\xi\oplus\eta))\ar[d]^{\Th} } \\
H_{\ast-2e}(\mathbb{S}\xi)\ar[r]_{\pi_\ast}  & H_{\ast-2e}(X)\ar[r]_-{c_{d-1}(\xi)\cap-}&  H_{\ast-2d-2e+2}(X)
 }\end{equation}
commutes.

The map $A\colon(\mathbb{B}(\mathbb{C}\oplus\pi^\ast\eta),\mathbb{S}(\mathbb{C}\oplus\pi^\ast\eta))\to (\mathbb{B}(\xi\oplus\eta),\mathbb{S}(\xi\oplus\eta))$ is just $\rho\circ(\mathrm{Id}\times j)$, but it is easy to see that it arises from the composite map of vector bundles
\[\xymatrix{
\mathbb{C}\oplus\pi^\ast\eta \ar[r]^{a\oplus 1}\ar[d] & \pi^\ast\xi\oplus\pi^\ast\eta\ar[d]\ar[r]^\pi & \xi\oplus\eta\ar[d]\\
\mathbb{S}\xi \ar@{=}[r] & \mathbb{S}\xi\ar[r]_\pi & X
}\]
where $a$ is the inclusion described at the start of the proof.  To show that \eqref{DiagramIII} commutes it therefore suffices to show that
\[\xymatrix{
H_{\ast+2}(\mathbb{B}(\mathbb{C}\oplus\pi^\ast\eta),\mathbb{S}(\mathbb{C}\oplus\pi^\ast\eta))\ar[r]^-{(a\oplus 1)_\ast}\ar[d]_\Th     & H_{\ast+2}(\mathbb{B}(\pi^\ast\xi\oplus\pi^\ast\eta),\mathbb{S}(\pi^\ast\xi\oplus\pi^\ast\eta)) \ar[d]^{\Th}  \\
H_{\ast-2e}(\mathbb{S}\xi) \ar[r]_-{\pi^\ast(c_{d-1}(\xi))\cap -} &  H_{\ast-2d-2e+2}(\mathbb{S}\xi)
}\]
commutes.  This is a standard result, so long as one replaces $\pi^\ast(c_{d-1}(\xi))$ in the above with the Euler class $E_{\mu}$, where $\mu$ is the complement to the sub-bundle $(a\oplus 1)(\mathbb{C}\oplus\pi^\ast\eta)$ of $\pi^\ast\xi\oplus\pi^\ast\eta$.  But this complement is just the bundle $\xi^\perp$ described at the start of the proof.  Its Euler class, by the definition of $c_{d-1}(\xi)$ explained there, is equal to $\pi^\ast(c_{d-1}(\xi))$.  This proves the theorem.
\end{proof}

\subsection{Computations with the Gysin sequence}\label{GysinSubsection}

In this subsection we will prove three simple lemmas to help with computations using the Gysin sequence
\begin{equation}\label{GysinSequenceEquation}\cdots\to H_\ast(\mathbb{S}\xi)\xrightarrow{\pi_\ast} H_\ast(X)\xrightarrow{E_\xi\cap-}H_{\ast-2d}(X)\to H_{\ast -1}(\mathbb{S}\xi)\to\cdots\end{equation}
associated to a rank $d$ complex vector bundle $\xi$ over $X$.  The lemmas are simple consequences of the construction of the Gysin sequence.  We include the proofs only for the sake of completeness.

Let us begin by recalling how to construct the Gysin sequence \eqref{GysinSequenceEquation}.  First form the pair $(\mathbb{B}\xi,\mathbb{S}\xi)$, where $\mathbb{B}\xi$ denotes the unit ball bundle of $\xi$, and take the associated long exact sequence of homology groups.  Then apply the isomorphisms
\begin{equation}\label{IsomorphismsEquation}H_\ast(\mathbb{B}\xi)\xrightarrow{\pi_\ast}H_\ast(X),\qquad H_\ast(\mathbb{B}\xi,\mathbb{S}\xi)\xrightarrow{\Th}H_{\ast-d}(X)\end{equation}
to transform this sequence into \eqref{GysinSequenceEquation}.

\begin{lemma}\label{GysinCapProductLemma}
Let $\xi\to X$ be a complex vector bundle of rank $d$ over $X$ and let $u\in H^\ast(X)$.  Then the diagram
\[\xymatrix{
\ar[r] & H_\ast(\mathbb{S}\xi)\ar[r]^{\pi_\ast}\ar[d]^{\pi^\ast u\cap -}  & H_\ast(X)\ar[r]^-{E_\xi\cap-}\ar[d]^{u\cap-} & H_{\ast-2d}(X)\ar[r]\ar[d]^{u\cap-} &  H_{\ast -1}(\mathbb{S}\xi)\ar[r]\ar[d]^{\pi^\ast u\cap -} & \\
\ar[r] & H_{\ast-|u|}(\mathbb{S}\xi)\ar[r]_{\pi_\ast}  & H_{\ast-|u|}(X)\ar[r]_-{E_\xi\cap-} & H_{\ast-2d-|u|}(X)\ar[r] &  H_{\ast -1-|u|}(\mathbb{S}\xi)\ar[r] & 
}\]
commutes up to sign.  To be precise, the first two squares commute strictly and the third commutes up to sign $(-1)^{|u|}$.
\end{lemma}
\begin{proof}
This is a simple consequence of the construction of the Gysin sequence, which we recalled earlier, and the formulas
\begin{eqnarray*}
u\cap(E_\xi\cap x)&=&E_\xi\cap (u\cap x),\\
\partial_\ast(\pi^\ast u\cap x)&=&(-1)^{|u|}\pi^\ast u\cap\partial_\ast x,\\
\pi^\ast u\cap(U_\xi\cap x)&=&U_\xi\cap (\pi^\ast u\cap x). 
\end{eqnarray*}
Here $U_\xi$ denotes the Thom class of $\xi$.
\end{proof}

\begin{lemma}\label{DirectSumGysinSequenceLemma}
Let $\xi$ and $\eta$ be complex vector bundles over a space $X$.  Then the Gysin sequences for $\xi$ and $\xi\oplus\eta$ are related by the following commutative diagram:
\[\xymatrix{
\ar[r] & H_\ast(\mathbb{S}\xi)\ar[r]^{}\ar[d]^{i_\ast}  & H_\ast(X)\ar[r]^-{}\ar@{=}[d] & H_{\ast-2d}(X)\ar[r]\ar[d]^{E_\eta\cap -} &  H_{\ast -1}(\mathbb{S}\xi)\ar[r]\ar[d]^{i_\ast} & \\
\ar[r] & H_\ast(\mathbb{S}(\xi\oplus\eta))\ar[r]_{}  & H_\ast(X)\ar[r]_-{} & H_{\ast-2d-2e}(X)\ar[r] &  H_{\ast -1}(\mathbb{S}(\xi\oplus\eta))\ar[r] &  
}\]
Here $i\colon\mathbb{S}\xi\hookrightarrow\mathbb{S}(\xi\oplus\eta)$ is the inclusion map, and $d=\rank\xi$ and $e=\rank\eta$.  A similar diagram relates the Gysin sequences for $\eta$ and $\xi\oplus\eta$, in which the vertical map $H_{\ast-q}(X)\to H_{\ast-p-q}(X)$ is just $E_\xi\cap-$.
\end{lemma}
\begin{proof}
There is an inclusion of pairs $k\colon(\mathbb{B}\xi,\mathbb{S}\xi)\hookrightarrow(\mathbb{B}(\xi\oplus\eta),\mathbb{S}(\xi\oplus\eta))$ and a corresponding commutative diagram of homology groups whose rows are the exact sequences of homology groups of the pairs  $(\mathbb{B}\xi,\mathbb{S}\xi)$ and $(\mathbb{B}(\xi\oplus\eta),\mathbb{S}(\xi\oplus\eta))$.  After applying the isomorphisms \eqref{IsomorphismsEquation} in both rows, we obtain the required commutative diagram, so long as we can verify that the vertical maps are of the form claimed.  This is clear except for the third vertical map.  However, by our assumptions we have that $k^\ast U_{\xi\oplus\eta}=U_\xi \cup \pi^\ast E_\eta$, so that $k^\ast U_{\xi\oplus\eta}\cap x=\pi^\ast E_\eta\cap(U_\xi\cap x)$.  Using this fact we can verify that the diagram
\[\xymatrix{
H_\ast(\mathbb{B}\xi,\mathbb{S}\xi)\ar[r]^-{U_{\xi}\cap-}\ar[d]^{k_\ast} & H_{\ast-p}(\mathbb{B}\xi)\ar[r]^{\pi_\ast}\ar[d]^{k_\ast(\pi^\ast E_\eta\cap -)} &  H_{\ast -p}(X)\ar[d]^{E_\eta\cap -} \\
H_\ast(\mathbb{B}(\xi\oplus\eta),\mathbb{S}(\xi\oplus\eta))\ar[r]_-{U_{\xi\oplus\eta}\cap-} & H_{\ast-p-q}(\mathbb{B}(\xi\oplus\eta))\ar[r]_-{\pi_\ast} &  H_{\ast-p-q}(X)
}\]
commutes.  The description of the third vertical map follows.  The proof of the second case is similar.
\end{proof}

\begin{lemma}\label{PullbackGysinSequenceLemma}
Let $f\colon X\to Y$ be a map of spaces and let $\xi$ be a complex vector bundle over $X$ of rank $d$.  Then the Gysin sequences for $\xi$ and $f^\ast\xi$ are related by the following commutative diagram:
\[\xymatrix{
\ar[r] & H_\ast(\mathbb{S}(f^\ast \xi))\ar[r]^{}\ar[d]^{}  & H_\ast(Y)\ar[r]^-{}\ar[d]^{f_\ast} & H_{\ast-2d}(Y)\ar[r]\ar[d]^{f_\ast} &  H_{\ast -1}(\mathbb{S}(f^\ast\xi))\ar[r]\ar[d]^{} & \\
\ar[r] & H_{\ast}(\mathbb{S}\xi)\ar[r]_{}  & H_{\ast}(X)\ar[r]_-{} & H_{\ast-2d}(X)\ar[r] &  H_{\ast -1}(\mathbb{S}\xi)\ar[r] & 
}\]
\end{lemma}
\begin{proof}
The construction of the Gysin sequence was recalled at the beginning of this subsection.  There is a map of pairs $\tilde f\colon (\mathbb{B}f^\ast\xi,\mathbb{S}f^\ast\xi)\to (\mathbb{B}\xi,\mathbb{S}\xi)$, which lies over $f\colon X\to Y$ and satisfies $\tilde f^\ast U_{\xi}=U_{f^\ast\xi}$.  The resulting commutative diagram, whose rows are the long exact sequences for $(\mathbb{B}f^\ast\xi,\mathbb{S}f^\ast\xi)$ and $(\mathbb{B}\xi,\mathbb{S}\xi)$, is then converted into the required diagram after applying the isomorphisms \eqref{IsomorphismsEquation}.
\end{proof}

\section{Statement of 
Theorem B and proof of 
Theorem A}\label{FirstReformulationSection}

In this section we will use the results of \S\ref{GroupActionsSection} to reformulate 
Theorem A.  Since $\CP^n$ is the space of lines in $\mathbb{C}^{n+1}$ it admits a natural action of $U(n+1)$.  Thus $\mathbb{H}_\ast(L\CP^n)$ is a module over both $H_\ast(U(n+1))$ and $H_\ast(\Omega U(n+1))$, and Theorem~\ref{ActionTheorem} explains how these module structures relate to the loop product and BV operator.  We begin by recalling the homology of $U(n+1)$ and of its loop space.  Throughout this section, and for the rest of the paper, homology groups are taken with coefficients in $\mathbb{Z}$.

\begin{definition}\label{RotationDefinition}
Let $R\colon S^1\times\CP^n\to U(n+1)$ be the map that sends a pair $(t,l)$ to the transformation that multiplies vectors in $l$ by ${t}$ and that leaves vectors orthogonal to $l$ unchanged.  (We are regarding $S^1$ as the unit complex numbers.)  Thus
\begin{equation}\label{RotationEquation}
R(t,l)\mathbf{v}=({t}-1)\langle \mathbf{l},\mathbf{v}\rangle\mathbf{l} + \mathbf{v}
\end{equation}
for $t\in S^1$, $l\in\CP^n$ and $\mathbf{v}\in\mathbb{C}^{n+1}$, where $\mathbf{l}$ is a unit vector in $l$.  Let $S\colon\CP^n\to\Omega U(n+1)$ denote the adjoint to $R$.
\end{definition}

\begin{definition}
Define $E_0,E_2,\ldots,E_{2n}$ in $H_\ast(\Omega U(n+1))$ by $E_{2i}=S_\ast[\CP^i]$ and define $e_1,e_3,\ldots,e_{2n+1}$ by $e_{2i+1}=R_\ast([S^1]\times[\CP^i])$.  It is well known that
\[H_\ast(\Omega U(n+1))=\mathbb{Z}[E_0,\ldots,E_{2n}], \qquad H_\ast(U(n+1))=\Lambda_\mathbb{Z}[e_1,\ldots,e_{2n+1}],\]
and that $e_{2i+1}=\sigma(E_{2i})$ and $D_\ast E_{2i}=\sum_{j=0}^i E_{2j}\otimes E_{2i-2j}$.
\end{definition}

We can now reformulate 
Theorem A in terms of the actions of $H_\ast(U(n+1))$ and $H_\ast(\Omega U(n+1))$ on $\mathbb{H}_\ast(L\CP^n)$.  This reformulation is necessarily rather detailed, but the crucial point is that the generator $v\in \mathbb{H}_{2n}(L\CP^n)$ can be chosen to be $E_{2n}1$.  As an algebra over $\mathbb{Z}[v]$ the BV algebra $\mathbb{H}_\ast(L\CP^n)$ is generated by $c$ and $w$, which themselves generate a finite dimensional subring.  Therefore, by using Theorem~\ref{ActionTheorem}, the task of computing $\Delta$ is reduced to the problem of computing the value of $\Delta$ and the effect of $H_\ast(U(n+1))$ and $H_\ast(\Omega U(n+1))$ on this finite dimensional subring.  In detail:

\begin{TheoremB}
The generators $c$, $w$, $v$ of $\mathbb{H}_\ast(L\CP^n)$ can be chosen so that:
\begin{enumerate}
\item
$E_{2i}1$ is equal to  $c^{n-i}\cdot v$ if $i>0$ and is equal to $1+c^n\cdot v$ if $i=0$.
\item
$e_{2j+1}c=0$ for all $j$, while $e_{2j+1}w= (j+1)c^{n-j}\cdot v$.
\item
$\Delta(c^i)=0$ and there are integers $\mu_{n-1},\ldots,\mu_0=1$ such that 
\[\Delta(c^i\cdot w)=\binom{n+1}{2} c^{n+i}\cdot v+ \mu_i c^i\]
for $i\geqslant 0$.  (Since $c^{n+1}\cdot v=0$, the first term is redundant unless $i=0$.)
\end{enumerate}
\end{TheoremB}

This theorem will be proved in \S\ref{SecondReformulationSection}.  The rest of this section is given to showing that 
Theorem A follows from it.  This follows by direct calculation from two simple lemmas: the first exploits the Batalin-Vilkovisky identity \eqref{BVIdentityEquation} to determine the $\mu_i$, and the second uses Theorem~\ref{ActionTheorem} to compute the effect of $\Delta$ on multiples of $v$.

\begin{lemma}\label{DeltaLemma}
$\mu_i=n-i$, so that
$\Delta(c^i\cdot w)=\binom{n+1}{2} c^{n+i}\cdot v+ (n-i) c^i$
for $i\geqslant 0$.
\end{lemma}
\begin{proof}
Introduce a new integer $\mu_n$ determined by the identity $\Delta(c^n\cdot w)=\mu_n c^n$.  Of course, since $c^n\cdot w=0$ we must have $\mu_n=0$.  Now take the Batalin-Vilkovisky identity \eqref{BVIdentityEquation} in the case $x=c^{i-1}$, $y=c$, $z=w$, with $1\leqslant i\leqslant n$ to find that
\begin{align*}
\Delta(c^{i}\cdot w)=&\Delta(c^{i})\cdot w + c^{i-1}\cdot \Delta(c\cdot w)+c\cdot\Delta(c^{i-1}\cdot w)\\
&-\Delta(c^{i-1})\cdot c\cdot w-\Delta(c)\cdot c^{i-1}\cdot w - c^{i}\cdot\Delta(w).
\end{align*}
Using the description of $\Delta$ given in the third part of 
Theorem B this equation becomes $\mu_{i}=\mu_{i-1}+\mu_1-\mu_0$
so that $\mu_i=i(\mu_1-\mu_0)+\mu_0$  for $1\leqslant i\leqslant n$.
Taking $i=(n-1)$, $i=n$ and using the fact that $\mu_{n-1}=1$ from the third part of 
Theorem B, and that fact that $\mu_n=0$ as remarked above, we can solve to find that $\mu_0=n$ and $\mu_1=(n-1)$.  Consequently $\mu_i=n-i$ as required.
\end{proof}

\begin{lemma}\label{PartialLemma}
Let us write $\partial_w\colon \mathbb{H}_\ast(L\CP^n)\to\mathbb{H}_{\ast+1}(L\CP^n)$ for the linear map that sends monomials $c^p\cdot v^q$ to $0$ and that sends monomials $c^p\cdot w\cdot v^q$ to $c^p\cdot v^q$.  Then for any $x\in\mathbb{H}_{\ast+1}(L\CP^n)$ we have
\[\Delta(v\cdot x) =\left[(n+1)v + \binom{n+1}{2} c^n\cdot v^2\right]\cdot\partial_w x + v\cdot\Delta(x).\]
\end{lemma}
\begin{proof}
By the second part of 
Theorem B we know the effect of $e_{2j+1}$ on $c$ and $w$, and by the first part of 
Theorem B we have that $e_{2j+1}v=e_{2j+1}E_{2n}1=E_{2n}(e_{2j+1}1)=0$.  By equation \eqref{GProductEquation} from Theorem~\ref{ActionTheorem} we therefore have that $e_{2i+1}x= (i+1) c^{n-i}\cdot v\cdot\partial_w x$.

Recall from the first part of 
Theorem B that $E_{2i}1=c^{n-i}\cdot v$ for $i>0$ and that $E_01=1+c^n\cdot v$.
Recall also that $D_\ast E_{2n} = \sum_{i=0}^n E_{2n-2i}\times E_{2i}$ and that $\sigma(E_{2i})=e_{2i+1}$.
Thus, by equation \eqref{OmegaDeltaEquation} of Theorem~\ref{ActionTheorem} we have
\begin{eqnarray*}
\Delta(v\cdot x)
&=& \Delta(E_{2n}x)\\
&=& \sum_{i=0}^n E_{2n-2i} e_{2i+1} x + E_{2n}\Delta(x)\\
&=& (1+c^n\cdot v)\cdot e_{2n+1} x + \sum_{i=0}^{n-1} c^i\cdot v\cdot  e_{2i+1} x + v\cdot\Delta(x)
\end{eqnarray*}
Now using our description of $e_{2i+1}x$ and the fact that $(n+1)c^n\cdot v=0$, this simplifies to give 
\begin{eqnarray*}
\Delta(v\cdot x)&=& (n+1)v\cdot\partial_w x + \sum_{i=0}^{n-1}(i+1) c^n\cdot v^2\cdot\partial_w x + v\cdot\Delta(x)\\
&=& \left[(n+1)v + \binom{n+1}{2} c^n\cdot v^2\right]\cdot\partial_w x + v\cdot\Delta(x)\\
\end{eqnarray*}
as required.
\end{proof}

\begin{proof}[Proof of 
Theorem A]
We begin by showing that $\Delta(c^p\cdot v^q)=0$ for all $p,q\geqslant 0$.  Note that $\partial_w(c^p\cdot v^{q})=0$ so that by Lemma~\ref{PartialLemma} we have  $\Delta(c^p\cdot v^{q+1}) = v\Delta(c^p\cdot v^{q})$.  Using this result and the third part of 
Theorem B, we have $\Delta(c^p\cdot v^q)=0$ by induction.

Now we compute $\Delta(c^p\cdot w\cdot v^q)$.  Note that $\partial_w(c^p\cdot w\cdot v^{q-1})=c^p\cdot v^{q-1}$, so that by Lemma~\ref{PartialLemma} we have
\begin{eqnarray*}
\Delta(c^p\cdot w\cdot v^q)
&=&\left[(n+1)v + \binom{n+1}{2} c^n\cdot v^2\right]\cdot c^p\cdot v^{q-1} + v\cdot\Delta(c^p\cdot w\cdot v^{q-1})\\
&=&\left[(n+1)c^p\cdot v^q + \binom{n+1}{2} c^{n+p}\cdot v^{q+1}\right]+ v\cdot\Delta(c^p\cdot w\cdot v^{q-1}).
\end{eqnarray*}
By induction we therefore have
\[\Delta(c^p\cdot w\cdot v^q)=q\left[(n+1)c^p\cdot v^q + \binom{n+1}{2} c^{n+p}\cdot v^{q+1}\right]+ v^q\cdot\Delta(c^p\cdot w)\]
and applying the description of $\Delta(c^p\cdot w)$ given by Lemma~\ref{DeltaLemma} this becomes
\begin{eqnarray*}
\Delta(c^p\cdot w\cdot v^q)
&=&q\left[(n+1)c^p\cdot v^q + \binom{n+1}{2} c^{n+p}\cdot v^{q+1}\right]\\
&&+ v^q\cdot \left[\binom{n+1}{2} c^{n+p}\cdot v+ (n-p) c^p\right]\\
&=& \left[q(n+1)+(n-p)\right] c^p\cdot v^q + (q+1)\binom{n+1}{2}c^{n+p}v^{q+1}
\end{eqnarray*}
as required.  This completes the proof.
\end{proof}

\section{A finite dimensional approximation to $L\CP^n$}\label{ApproximationSection}

We have seen in \S\ref{FirstReformulationSection} that Theorem A follows from Theorem B.  Our goal now is to prove Theorem B. This is a finite dimensional result, in the sense that it describes the effect of $H_\ast(U(n+1))$, $H_\ast(\Omega U(n+1))$ and $\Delta$ on the finite dimensional subring of $\mathbb{H}_\ast(L\CP^n)$ generated by $c$ and $w$.  Our strategy is to rephrase Theorem B as a result about a finite dimensional space that we can understand in explicit terms.

We begin this strategy in the present section.  In \S\ref{LnConstructionSubsection} we construct a finite dimensional space $\mathbb{L}^n$ and a map $\mathbb{L}^n\to L\CP^n$.  Then in \S\ref{gInjectiveSubsection} we see that this map induces an \emph{injection} $H_\ast(\mathbb{L}^n)\to H_\ast(L\CP^n)$, so that we may regard $\mathbb{L}^n$ as approximating $L\CP^n$ in low dimensions.  In \S\ref{SecondReformulationSection} we will see that 
Theorem B can be reformulated in terms of $\mathbb{L}^n$.  Again, all homology groups are taken with coefficients in $\mathbb{Z}$.

\subsection{The space $\mathbb{L}_n$}\label{LnConstructionSubsection}

The free loop space $L\CP^n$ has the following structure:
\begin{enumerate}
\item The projection map $L\CP^n\to \CP^n$ that sends a loop to its basepoint.
\item The action of $U(n+1)$ induced from the natural action of $U(n+1)$ on $\CP^n$.
\item The circle action given by rotating loops.
\end{enumerate}
The two actions commute, and the projection map is equivariant with respect to the $U(n+1)$ actions.
Our aim in this subsection is to produce a finite dimensional space $\mathbb{L}^n$ equipped with analogous structure \emph{and} a map into $L\CP^n$ that preserves the structure.  We begin by giving an intermediate construction, based on the generating map $S\colon\CP^n\to\Omega U(n+1)$ described in Definition~\ref{RotationDefinition}.

\begin{definition}
We equip $\CP^n\times\CP^n$ with the following structure:
\begin{enumerate}
\item The projection $\CP^n\times\CP^n\to\CP^n$ onto the second factor.
\item The $U(n+1)$ action where $U(n+1)$ acts simultaneously on each factor in the natural way.
\item The circle action defined by $t(l,m)=(l,R(t,l)m)$ for $t\in S^1$ and $(l,m)\in\CP^n\times\CP^n$.  Here $R\colon S^1\times\CP^n\to U(n+1)$ is the rotation map of Definition~\ref{RotationDefinition}.
\end{enumerate}
\end{definition}

\begin{definition}\label{fDefinition}
We define
\[f\colon \CP^n\times\CP^n\to L\CP^n\]
to be the map that sends $(l,m)\in\CP^n\times\CP^n$ to the loop $t\mapsto R(t,l)m$.  The significance of $f$ is that it can be regarded as the composite 
\[\CP^n\times\CP^n\xrightarrow{S\times c}\Omega U(n+1)\times L\CP^n\to L\CP^n\]
where $S\colon\CP^n\to \Omega U(n+1)$ is the generating map of Definition~\ref{RotationDefinition} and $c\colon\CP^n\to L\CP^n$ is the inclusion of the constant loops.  Since $S_\ast[\CP^i]$ is $E_{2i}$ and $c_\ast[\CP^n]$ is the unit of $\mathbb{H}_\ast(L\CP^n)$, $f_\ast([\CP^i]\times[\CP^n])$ is in fact the element $E_{2i}1$ that we wish to identify in the first part of 
Theorem B.
\end{definition}

\begin{lemma}
The $S^1$ and $U(n+1)$ actions on $\CP^n\times\CP^n$ commute.  The map $f\colon\CP^n\times\CP^n\to L\CP^n$ intertwines the $S^1$ and $U(n+1)$ actions and commutes with the projections to $\CP^n$.
\end{lemma}
\begin{proof}
That the actions commute and that $f$ is equivariant follow from the fact that for $l\in\CP^n$ we have $AR(t,l)A^{-1}=R(t,Al)$, $R(t_1,l)R(t_2,l)=R(t_1+t_2,l)$ for $A\in U(n+1)$ and $t_1,t_2\in S^1$.  The final property follows from the fact that $R(0,l)$ is the identity element.
\end{proof}

\begin{definition}\label{LnDefinition}
Given $l\in\CP^n$, let $\CP^{n-1}_{l^\perp}\subset\CP^n$ be the subset of lines perpendicular to $l$.  We define $\mathbb{L}^n$ to be the space obtained from $\CP^n\times\CP^n$ by collapsing, for each $l\in\CP^n$, the subset $\CP^{n-1}_{l^\perp}\times\{l\}\sqcup \{l\}\times\{l\}$ to a point $p_{l}$.  Note that the various points $p_l$ are not identified with one another.
\end{definition}

\begin{proposition}\label{LnStructuresExistProposition}
\begin{enumerate}
\item
The projection $\CP^n\times\CP^n\to \CP^n$ reduces to a projection $\mathbb{L}^n\to\CP^n$ and the actions of $U(n+1)$ and $S^1$ on $\CP^n\times\CP^n$ reduce to actions of $U(n+1)$ and $S^1$ on $\mathbb{L}^n$.
\item
The map $f\colon\CP^n\times\CP^n\to L\CP^n$ reduces to a map $g\colon\mathbb{L}^n\to L\CP^n$ that intertwines the $U(n+1)$ and $S^1$ actions and that commutes with the projections to $\CP^n$.
\end{enumerate}
\end{proposition}
\begin{proof}
Note that:
\begin{enumerate}
\item The projection $\CP^n\times\CP^n\to\CP^n$ sends each subset $\CP^{n-1}_{l^\perp}\times\{l\}\sqcup\{l\}\times\{l\}$ to the single point $l$, and so reduces to a projection map $\mathbb{L}^n\to\CP^n$.
\item An element $A\in U(n+1)$ sends $\CP^{n-1}_{l^\perp}\times\{l\}\sqcup\{l\}\times\{l\}$ into $\CP^{n-1}_{Al^\perp}\times\{Al\}\sqcup\{Al\}\times\{Al\}$, so that the $U(n+1)$ action on $\CP^n\times\CP^n$ reduces to a $U(n+1)$ action on $\mathbb{L}^n$.
\item The circle action on $\CP^n\times\CP^n$ fixes each subset $\CP^{n-1}_{l^\perp}\times\{l\}\sqcup\{l\}\times\{l\}$ pointwise, and so reduces to a circle action on $\mathbb{L}^n$.
\end{enumerate}
The first two statements are completely clear.  The final statement relies on the fact that the operation $R(t,m)$, which multiplies vectors in $m$ by the complex number $t$ and preserves vectors orthogonal to $m$, sends a line $l$ into itself when $l=m$ and when $l$ and $m$ are perpendicular.  The same fact guarantees that $f\colon\CP^n\times\CP^n\to L\CP^n$, which sends $(m,l)$ to  the loop $t\mapsto R(t,m)l$, sends the subset $\{l\}\times\{l\}\sqcup\CP^{n-1}_{l^\perp}\times\{l\}$ to the constant loop based at $l$.  It follows that $f$ reduces to the required map $g$.  The three required properties of $g$ follow from the same properties of $f$.
\end{proof}

\subsection{The map $g_\ast\colon\mathbb{L}^n\to L\CP^n$}\label{gInjectiveSubsection}
We have now obtained a space $\mathbb{L}^n$, equipped with three structures analogous to the three structures on $L\CP^n$ described at the start of \S\ref{LnConstructionSubsection}, and that is equipped with a map to $L\CP^n$ that intertwines these structures.  This is our desired approximation of $L\CP^n$.  Our use of the word `approximation' is justified by the following proposition.

\begin{proposition}\label{gInjectiveProposition}
The map $g_\ast\colon H_\ast(\mathbb{L}^n)\to H_\ast(L\CP^n)$ is split injective.
\end{proposition}
\begin{proof}
Write $l_0\in\CP^n$ for the span of the first standard basis vector.  There is a commutative diagram
\begin{equation}\label{FibrationsDiagram}\xymatrix{
V_n\ar[r]\ar[d]_{\gamma}  & \mathbb{L}^n\ar[r]^\pi\ar[d]_g & \CP^n\ar@{=}[d]\\
\Omega_{l_0}\CP^n \ar[r] & L\CP^n \ar[r]_\pi & \CP^n.
}\end{equation}
Here $V_n$ is the space obtained from $\CP^n$ by identifying the subset $\CP^{n-1}_{l_0^\perp}\times\{l_0\}\sqcup\{l_0\}\times\{l_0\}$ to a single point.  The maps $\pi$ are the projections of fibre bundles.  For the bottom row of the diagram this is well known.  For the top row this is reasonably clear, but will in any case be proved in Proposition~\ref{LnIsACompactificationProposition}.

We note that $H_\ast(V_n)$ vanishes in degrees greater than $2n$.  Further, we claim that $\gamma_\ast\colon H_\ast(V_n)\to H_\ast(\Omega_{l_0}\CP^n)$ is an injection that is isomorphic in degrees up to and including $2n$.  This is proved in Lemma~\ref{LoopHomologyInjective} below.  The proposition now follows from the routine spectral sequence argument that we present below.

Let $(E^r,d_E^r)$ denote the Serre spectral sequence in homology for the top row of \eqref{FibrationsDiagram} and let $(F^r,d_F^r)$ denote the Serre spectral sequence for the bottom row of \eqref{FibrationsDiagram}.  Thus $E^2_{\ast,\ast}=H_\ast(\CP^n)\otimes H_\ast(V_n)$ and $F^2_{\ast,\ast}=H_\ast(\CP^n)\otimes H_\ast(\Omega_{l_0}\CP^n)$.  There is an induced map $g_\ast\colon E^r\to F^r$ of spectral sequences in which $g_\ast\colon E^2\to F^2$ is given by $\mathrm{Id}\otimes\gamma_\ast$.  Thus $E^2$ is concentrated in bidegrees $(i,j)$ with $j\leqslant 2n$ and $g_\ast\colon E^2\to F^2$ is an isomorphism in these bidegrees.  The same property then holds for $E^r$ and $g_\ast\colon E^r\to F^r$ for all $r\geqslant 2$, and consequently for $E^\infty$ and $g_\ast\colon E^\infty\to F^\infty$.   

Now fix a degree $k\geqslant 0$.  Let $A_0\leqslant A_1\leqslant \cdots \leqslant A_n$ denote the filtration of $H_k(\mathbb{L}^n)$ induced by the CW filtration of $\CP^n$.  Let $B_0\leqslant B_1\leqslant\cdots\leqslant B_n$ denote the filtration of $H_k(L\CP^n)$ induced by the CW filtration of $\CP^n$.  The map $g_\ast\colon H_\ast(\mathbb{L}^n)\to H_\ast(L\CP^n)$ respects these filtrations and induces maps $A_i/A_{i-1}\to B_i/B_{i-1}$ that under the identifications $A_i/A_{i-1}\cong E^\infty_{i,k-i}$, $B_i/B_{i-1}\cong F^\infty_{i,k-i}$ correspond to $g_\ast\colon E^\infty\to F^\infty$.  

We have established that $g_\ast$ induces isomorphisms $A_i/A_{i-1} \to B_i/B_{i-1}$ for $i\geqslant k-2n$ so that it therefore induces an isomorphism $A_{2n}/A_{k-2n-1}\to B_{2n}/B_{k-2n-1}$.  We have also established that $A_0=\cdots=A_{k-2n-1}=0$ so that $A_{2n}/A_{k-2n-1}=H_k(\mathbb{L}^n)$ and therefore this last isomorphism has the form $H_{k}(\mathbb{L}^n){\cong} H_k(L\CP^n)/B_{k-2n-1}$.
This gives us the required splitting $H_k(L\CP^n)\to (H_k(L\CP^n)/B_{k-2n-1})\cong H_k(\mathbb{L}^n)$ of $g_\ast$.  This completes the proof.
\end{proof}

\begin{lemma}\label{LoopHomologyInjective}
$\gamma_\ast\colon H_\ast(V_n)\to H_\ast(\Omega_{l_0}\CP^n)$ is an isomorphism in degrees $\ast\leqslant 2n$.
\end{lemma}
\begin{proof}
Recall that $V_n$ consists of $\CP^n$ with the subset $\CP^{n-1}_{l_0^\perp}\times\{l_0\}\sqcup\{l_0\}\times\{l_0\}$ identified to a single point and recall from Proposition~\ref{LnStructuresExistProposition} that $\gamma$ is the map that makes the triangle 
\[\xymatrix{
\CP^n \ar[r]^-S\ar[dr]_{q} & \Omega U(n+1) \ar[r]^{\Omega\alpha}  & \Omega_{l_0}\CP^n \\
{} & V_n\ar[ur]_\gamma & {}
}\]
commute.  Here $S$ is the map of Definition~\ref{RotationDefinition}, $\alpha\colon U(n+1)\to \CP^n$ is $A\mapsto Al_0$, and $q$ is the identification map.

The homology groups of $V_n$ are zero except in degrees $0$, $1$ and $2n$.  In degree $2n$ a generator is given by $q_\ast[\CP^n]$.  In degree $1$ we form the infinite closed interval $[0,\infty]$, which maps into $\CP^n$ by $r\mapsto [1,r,0,\ldots,0]$.  The initial point of this map is $l_0$ and the terminal point is $l_1\in\CP^{n-1}$.  Applying $q$ we obtain a path $\delta\colon[0,\infty]\to V_n$ that identifies $0$ and $\infty$.  The corresponding cycle in $V_n$ represents a generator of $H_1(V_n)$ that we denote by $[\delta]$.

The homology ring $H_\ast(\Omega_{l_0}\CP^n)$ is $\mathbb{Z}[a_1,a_{2n}]/\langle a_1^2\rangle$, where $a_1$ and $a_{2n}$ lie in degrees $1$ and $2n$ respectively.  Here $a_1\in H_1(\Omega_{l_0}\CP^n)$ is the unique class whose homology suspension $\sigma(a_1)\in H_2(\CP^n)$ is the fundamental class $[\CP^1]$, and it is not hard to see that $a_{2n}\in H_{2n}(\Omega_{l_0}\CP^n)$ can be taken to be the image of $E_{2n}$ under $\Omega\alpha_\ast$.

We will prove the proposition by showing that $\gamma_\ast[\delta]=a_1$ and that $\gamma_\ast q_\ast[\CP^n]=a_{2n}$.  The second of these is immediate from the commutative diagram above.  To prove the first we will show that the homology suspension $\sigma([\delta])$ is equal to $[\CP^1]$.  The homology suspension $\sigma([\delta])$ is represented by the cycle $S^1\times[0,\infty]\to\CP^1\hookrightarrow\CP^n$, where the first map sends $(t,r)$ to
\[\left[\frac{t+r^2}{1+r^2},\frac{r(t-1)}{1+r^2}\right].\]
We must therefore show that the degree of $S^1\times[0,\infty]\to\CP^1$ is $1$.  To see this, note that $[0,1]$ has a single preimage $(-1,1)$.   For $t\neq 1$ the map sends $(t,r)$ to
\[\left[\frac{t+r^2}{r(t-1)},1\right]\]
and so the derivative of $S^1\times[0,\infty]\to\CP^1$ at $(-1,1)$ is the map $\mathbb{R}^2\to\mathbb{C}=T_{[0,1]}\CP^1$ that sends $\partial/\partial t$ to $i/2$ and $\partial/\partial r$ to $-1$.  This is an isomorphism that preserves orientations, and so the degree is $+1$ as claimed.
\end{proof}

\section{Statement of Theorem C and proof of Theorem B}\label{SecondReformulationSection}

In \S\ref{ApproximationSection} we constructed a finite-dimensional approximation $\mathbb{L}^n$ of $L\CP^n$.  In this section we will state Theorem C, which describes the homology of $\mathbb{L}^n$, and use it to prove Theorem B.  Homology groups are taken with coefficients in $\mathbb{Z}$.

\begin{definition}
Let $\gamma_n$ denote the tautological complex line bundle on $\CP^n$.  Let $u\in H^2(\CP^n)$ denote the first Chern class of $\gamma_n^\ast$.  For any space $X$ equipped with a map $\pi\colon X\to\CP^n$, the pullback $\pi^\ast u\in H^2(X)$ will also be denoted by $u$.  In particular, we have classes $u\in H^2(\mathbb{L}^n)$ and $u\in H^2(L\CP^n)$.
\end{definition}

Consider the chain of standard inclusions $\CP^0\subset\CP^1\subset\cdots\subset\CP^n$, where each space is regarded as the collection of points in the next whose final homogeneous coordinate vanishes.  The fundamental classes $[\CP^0],\ldots,[\CP^n]$ form a basis for $H_\ast(\CP^n)$.  It is well known that the normal bundle of $\CP^{n-1}$ in $\CP^n$ is $\gamma_{n-1}^\ast$, and from this it follows that $u\cap[\CP^i] = [\CP^{i-1}]$ for $1\leqslant i\leqslant n$.

Recall from Proposition~\ref{LnStructuresExistProposition} that $\mathbb{L}^n$ admits a projection to $\CP^n$ and actions of $U(n+1)$ and $S^1$.  Therefore $H_\ast(\mathbb{L}^n)$ admits a homomorphism to $H_\ast(\CP^n)$ and an action of $H^\ast(\CP^n)$, an action of $H_\ast(U(n+1))$, and a degree-raising operator associated to the action of $S^1$.  Also, since $\mathbb{L}^n$ was constructed as a quotient of $\CP^n\times\CP^n$ there is a homomorphism $H_\ast(\CP^n\times\CP^n)\to H_\ast(\mathbb{L}^n)$.

\begin{TheoremC}
$H_\ast(\mathbb{L}^n)$ has generators
\[[\CP^0],\ldots,[\CP^n],\quad a_0,\ldots,a_{n-1},\quad b_0,\ldots,b_n\]
with degrees $|[\CP^i]|=2i$,  $|a_i|=2i+1$,  $|b_i|=2i+2n$.  These generators are subject to the single relation $(n+1)b_0=0$.  Moreover:
\begin{enumerate}
\item
The projection $H_\ast(\mathbb{L}^n)\to H_\ast(\CP^n)$ sends $[\CP^i]$ to $[\CP^i]$ and sends the $a_i$ and $b_i$ to $0$.
\item
For $i\geqslant 0$ the cap product with $u$ sends $[\CP^{i+1}]$ to $[\CP^i]$, $a_{i+1}$ to $a_i$, and $b_{i+1}$ to $b_i$.  It annihilates $[\CP^0]$, $a_0$ and $b_0$.
\item
The action of $H_\ast(U(n+1))$ on $H_\ast(\mathbb{L}^n)$ satisfies
\[e_{2j+1}b_n=0,\qquad e_{2j+1}[\CP^n]=0,\qquad e_{2j+1}a_{n-1}=(j+1)b_{j} \]
for all $j\geqslant 0$.
\item
The degree-raising operator associated to the action of $S^1$ on $\mathbb{L}^n$ can be described as follows.  It annihilates the $[\CP^i]$ and $b_i$, and there are integers $\mu_{n-1},\ldots,\mu_1,\mu_0=1$ such that 
\[a_i\mapsto\mu_i[\CP^{i+1}]\]
for $0\leqslant i<(n-1)$ and such that
\[a_{n-1}\mapsto\binom{n+1}{2} b_0+\mu_{n-1}[\CP^n].\]
\item
The homomorphism $H_\ast(\CP^n\times\CP^n)\to H_\ast(\mathbb{L}^n)$ sends $[\CP^i]\times[\CP^n]$ to $b_{i}$ if $i>0$ and to $b_0+[\CP^n]$ if $i=0$.
\end{enumerate}
\end{TheoremC}

\begin{proof}[Proof of 
Theorem B]
To begin, let us recall from Proposition~\ref{LnStructuresExistProposition}  that we have a map
\[g_\ast\colon H_\ast(\mathbb{L}^n)\to H_\ast(L\CP^n)\]
that intertwines the $H_\ast(S^1)$ and $H_\ast(U(n+1))$ actions, and that satisfies $u\cap g_\ast(x)=g_\ast(u\cap x)$ for any $x\in H_\ast(\mathbb{L}^n)$.  Moreover, $g_\ast$ is split-injective by Proposition~\ref{gInjectiveProposition}.  Finally, the map $f\colon\CP^n\times\CP^n\to L\CP^n$ from Definition~\ref{fDefinition}, which describes the classes $E_{2i}1=f_\ast([\CP^i]\times[\CP^n])$, factors as $g\circ q$ where $q\colon\CP^n\times\CP^n\to\mathbb{L}^n$ denotes the quotient map.

Theorem C shows that $[\CP^{n-1}]$, $a_{n-1}$ and $b_n$ generate free summands of $H_\ast(\mathbb{L}^n)$ in degrees $2n-2$, $2n-1$ and $4n$ respectively, so that $g_\ast[\CP^{n-1}]$, $g_\ast a_{n-1}$ and $g_\ast b_n$ generate free summands of $H_\ast(L\CP^n)$ in degrees $2n-2$, $2n-1$ and $4n$ respectively.  Shifting degrees we find that $g_\ast[\CP^{n-1}]$, $g_\ast a_{n-1}$ and $g_\ast b_n$ generate free summands of $\mathbb{H}_\ast(L\CP^n)$ in degrees $-2$, $-1$ and $2n$ respectively.  As explained in the introduction, the generators $c$, $w$ and $v$ of $\mathbb{H}_\ast(L\CP^n)$ may be chosen to be any classes that generate a free summand in the relevant degrees.  We may therefore take
\[ c=g_\ast[\CP^{n-1}],\quad w=g_\ast a_{n-1},\quad v= g_\ast b_n.\]

The first part of 
Theorem C shows that $c$ is the image of $[\CP^{n-1}]$ under the inclusion $H_\ast(\CP^n)\to H_\ast(L\CP^n)$.  We also know that  $[\CP^{n-1}]=u\cap[\CP^n]$.  In other words, $c=u\cap 1$ in $\mathbb{H}_\ast(L\CP^n)$.  By a theorem of Tamanoi  \cite[Theorem A]{\TamanoiCap} we therefore have $c\cdot x= u\cap x$ for any $x\in\mathbb{H}_\ast(L\CP^n)$.  Therefore, using the description of the map $u\cap -$ on $H_\ast(\mathbb{L}^n)$ given by the second part of 
Theorem C,  we have
\[ g_\ast [\CP^{n-i}] = c^{i},\quad g_\ast a_{n-i-1} =  c^{i}\cdot w,\quad g_\ast b_{n-i} = c^{i}\cdot v.\]

We can now prove the theorem.  First, $E_{2i}1 = f_\ast([\CP^i]\times[\CP^n]) = g_\ast\circ q_\ast([\CP^i]\times[\CP^n])$, which by the final part of 
Theorem C is equal to $g_\ast(b_i)=c^{n-i}\cdot v$ if $i>0$ and is equal to $g_\ast([\CP^n]+ b_0)=1+c^n\cdot v$ if $i=0$.  Second, since $g_\ast$ is a $H_\ast(U(n+1))$ module homomorphism, the third part of 
Theorem C shows that $e_{2j+1}v=0$, $e_{2j+1}c=0$ and $e_{2j+1}w=(j+1)c^{n-j}\cdot v$.  Finally, since $g_\ast$ is a map of $H_\ast(S^1)$ modules we have that $g_\ast(R x)=\Delta(g_\ast x)$ for any $x\in H_\ast(\mathbb{L}^n)$.   The fourth part of 
Theorem C now gives us $\Delta(c^i)=0$ and $\Delta(c^i\cdot v)=0$, and integers $\mu_{n-1},\ldots,\mu_0=1$ such that $\Delta(c^i\cdot w)=\mu_i c^i$ for $i>0$ and
\[\Delta(w)=\binom{n+1}{2} c^n\cdot v+\mu_{n-1}1.\]
This completes the proof.
\end{proof}

\section{$\mathbb{L}^n$ as a compactification of $T\CP^n$}\label{LnIsACompactificationSection}
We have established that Theorem A follows from Theorem C, which describes the homology of the space $\mathbb{L}^n$ introduced in \S\ref{ApproximationSection}.  We now work towards a proof of Theorem C.   In \S\ref{LnPinchedSubsection} we see that $\mathbb{L}^n$ can be identified with the pinched compactification $(T\CP^n)^\mathrm{P}$.  Then in \S\ref{LnPinchedComputationsSubsection} we use the techniques of \S\ref{CompactificationsSection} to compute the homology of $\mathbb{L}^n$ in terms of $\mathbb{S}T\CP^n$.

\subsection{$\mathbb{L}^n$ and $(T\CP^n)^\mathrm{P}$}\label{LnPinchedSubsection}

$\mathbb{L}^n$ was constructed in Definition~\ref{LnDefinition} as the quotient of $\CP^n\times\CP^n$ obtained by identifying, for each $l\in\CP^n$, the subset $\{l\}\times\{l\}\sqcup\CP^{n-1}_{l^\perp}\times\{l\}$ to a single point.  In this section we will identify $\mathbb{L}^n$ as the pinched compactification of $T\CP^n$ introduced in \S\ref{CompactificationsSection}.  We begin with a concrete description of $T\CP^n$.

\begin{note}
Recall that $\gamma_n\to\CP^n$ denotes the tautological line bundle and that $u\in H^2(\CP^n)$ denotes the first Chern class of $\gamma_n^\ast$.
By regarding $\gamma_n$ as a sub-bundle of the trivial bundle $\mathbb{C}^{n+1}\to\CP^n$ we may form its orthogonal complement $\gamma_n^\perp\to \CP^n$.  Recall that $T\CP^n =\gamma_n^\perp\otimes \gamma_n^\ast$.  From this it is easy to see that
\[c(T\CP^n) = (1+u)^{n+1}.\]
Given a nonzero vector $\mathbf{u}\in\mathbb{C}^{n+1}$, we will write $[\mathbf{u}]\in\CP^n$ for the line spanned by $\mathbf{u}$.
We can describe $T\CP^n$ as the quotient space
\begin{equation}\label{AlternativeTCPnEquation}
T\CP^n=\left\{(\mathbf{u},\mathbf{v})\in\mathbb{C}^{n+1}\times\mathbb{C}^{n+1}\mid\|\mathbf{u}\|=1,\langle \mathbf{u},\mathbf{v}\rangle=0\right\}\big\slash S^1
\end{equation}
where $S^1$ acts by the formula $t(\mathbf{u},\mathbf{v})=({t}\mathbf{u}, {t}\mathbf{v})$.  We write the equivalence class of a pair $(\mathbf{u},\mathbf{v})$ as $[\mathbf{u},\mathbf{v}]$.  The projection $T\CP^n\to\CP^n$ is given by $[\mathbf{u},\mathbf{v}]\mapsto[\mathbf{u}]$.  The isomorphism of this space with $\gamma_n^\perp\otimes\gamma_n^\ast$ sends $[\mathbf{u},\mathbf{v}]$ to $\mathbf{v}\otimes\mathbf{u}^\ast\in(\gamma_n^\perp\otimes\gamma_n^\ast)_{[\mathbf{u}]}$.  Here $\mathbf{u}^\ast$ denotes the functional $\langle\mathbf{u},-\rangle$.
\end{note}

\begin{proposition}\label{LnIsACompactificationProposition}
There is a diffeomorphism of almost-complex manifolds
\[\phi\colon\CP^n\times\CP^n\to (T\CP^n)^\mathbb{P}\]
over $\CP^n$ with the following properties:
\begin{enumerate}
\item The pullback $\phi^\ast\tau_{T\CP^n}$ is the exterior tensor product $\gamma_n\boxtimes\gamma_n^\ast$.
\item The point $(l,l)\in\CP^n\times\CP^n$ is sent by $\phi$ to the origin in $T_l\CP^n$.
\item The subset $\CP^{n-1}_{l^\perp}\times\{l\}$ is sent by $\phi$ to the subset $\mathbb{P}(T_l\CP^n)$ of $(T\CP^n)^\mathbb{P}$.
\item The action of $U(n+1)$ on $\CP^n\times\CP^n$ corresponds under $\phi$ to the action on $(T\CP^n)^\mathbb{P}$ that is induced from the natural action of $U(n+1)$ on $T\CP^n$.
\item Under $\phi$ the circle action on $\CP^n\times\CP^n$ corresponds to a circle action on $(T\CP^n)^\mathbb{P}$.  This circle action restricts to the action on $T\CP^n$ given by
\begin{equation}\label{CircleActionEquation}
t[\mathbf{u},\mathbf{v}]=\left[\frac{t+\|\mathbf{v}\|^2}{1+\|\mathbf{v}\|^2}\mathbf{u}+\frac{t-1}{1+\|\mathbf{v}\|^2}\mathbf{v},\frac{(t-1)\|\mathbf{v}\|^2}{1+\|\mathbf{v}\|^2}\mathbf{u}+\frac{{t}\|\mathbf{v}\|^2+1}{1+\|\mathbf{v}\|^2}\mathbf{v}\right].
\end{equation}
\end{enumerate}
\end{proposition}

\begin{corollary}\label{LnCorollary}
The diffeomorphism $\phi\colon\CP^n\times\CP^n\to(T\CP^n)^\mathbb{P}$ reduces to a homeomorphism $\psi\colon\mathbb{L}^n\to(T\CP^n)^\mathrm{P}$.  This map commutes with the projections to $\CP^n$, is equivariant with respect to the $U(n+1)$ actions, and transports the $S^1$ action on $\mathbb{L}^n$ to an action on $(T\CP^n)^\mathrm{P}$ that restricts to the action on $T\CP^n$ described by \eqref{CircleActionEquation}.
\end{corollary}

\begin{proof}[Proof of Proposition~\ref{LnIsACompactificationProposition}]
We begin with a general note about projectivisations of complex vector bundles.
Let $\xi$ be a complex vector bundle over a space $X$.  Then the associated bundle $\mathbb{P}\xi\to X$ of complex projective spaces admits a tautological complex line bundle $\delta_\xi\to\mathbb{P}\xi$.  A point in $\mathbb{P}\xi$ consists of a line $l$ in a fibre $\xi_x$, and the fibre of $\delta_\xi$ over this point is nothing but the line $l$.
If $\gamma$ is a complex line bundle over $X$ then there is a natural isomorphism $\phi\colon\mathbb{P}\xi\to \mathbb{P}(\xi\otimes\gamma)$ that sends $l\subset\xi_x$ to $l\otimes\gamma_x\subset \xi_x\otimes\gamma_x=(\xi\otimes\gamma)_x$.  Note that $\phi^\ast \delta_{\xi\otimes\gamma}=\delta_\xi\otimes\pi^\ast\gamma$, where $\pi\colon\mathbb{P}\xi\to X$ is the projection.

Now $\CP^n\times\CP^n$ is the projectivisation of the trivial bundle $\mathbb{C}^{n+1}$ over $\CP^n$, while $T\CP^n$ is the vector bundle $\gamma_n^\perp\otimes\gamma_n^\ast$.  We therefore have an isomorphism of vector bundles
\begin{equation}\label{VectorBundleIsomorphism}\mathbb{C}^{n+1}\otimes\gamma_n^\ast\cong T\CP^n\oplus\mathbb{C}
\end{equation}
and, as in the last paragraph, an isomorphism
\begin{equation}\label{phiDefinition}
\phi\colon \CP^n\times\CP^n = \mathbb{P}(\mathbb{C}^{n+1})\xrightarrow{\cong}\mathbb{P}(T\CP^n\oplus\mathbb{C})=(T\CP^n)^\mathbb{P}
\end{equation}
as required.  Definition~\ref{TauDefinition} defined a tautological line bundle $\tau_\xi$ on $\xi^\mathbb{P}$ for any $\xi\to X$.  But $\xi^\mathbb{P}=\mathbb{P}(\xi\oplus\mathbb{C})$, and one can see that $\tau_\xi=\delta_{\xi\oplus\mathbb{C}}$.  Now $\tau_{T\CP^n}=\delta_{T\CP^n\oplus\mathbb{C}}$ so that $\phi^\ast\tau_{T\CP^n}=\delta_{\mathbb{C}^{n+1}}\otimes\pi^\ast\gamma_n^\ast$.  But $\delta_{\mathbb{C}^{n+1}}=\gamma_n\boxtimes\mathbb{C}$ and $\pi^\ast\gamma_n^\ast=\mathbb{C}\boxtimes\gamma_n^\ast$.  This proves the first assertion.

We now prove the second part.  Under the identification $\CP^n\times\CP^n=\mathbb{P}(\mathbb{C}^{n+1})$ the point $(l,l)$ is just the span of $l\subset\mathbb{C}^{n+1}$ in the fibre over $l$.  In other words, it is the image of $l$ under the inclusion $\CP^n=\mathbb{P}(\gamma_n)\hookrightarrow\mathbb{P}(\mathbb{C}^{n+1})$.  Similarly, the embedding $T\CP^n\hookrightarrow(T\CP^n)^\mathbb{P}$ sends $0\in T_l\CP^n$ to the image of $l$ under the embedding $\CP^n=\mathbb{P}(\mathbb{C})\hookrightarrow\mathbb{P}(T\CP^n\oplus\mathbb{C})=(T\CP^n)^\mathbb{P}$.  To prove the second assertion it therefore suffices to note that, as in the first paragraph, there is a diffeomorphism $\phi\colon\mathbb{P}(\gamma_n)\to\mathbb{P}(\mathbb{C})$ that fits into a commutative diagram
\[\xymatrix{
\mathbb{P}(\gamma_n)\ar[r]^{\phi}\ar@{^(->}[d]  & \mathbb{P}(\mathbb{C})\ar@{^(->}[d]\\
\mathbb{P}(\mathbb{C}^{n+1})\ar[r]_-\phi& \mathbb{P}(T\CP^n\oplus\mathbb{C}).
}\]
Thus $\phi$ does indeed identify the required points.

We now prove the third part.  This is similar to the proof of the second part.  Under the identification $\CP^n\times\CP^n=\mathbb{P}(\mathbb{C}^{n+1})$ the subset $\CP^{n-1}_{l^\perp}\times\{l\}$ is the image under $\mathbb{P}(\gamma_n^\perp)\hookrightarrow\mathbb{P}(\mathbb{C}^{n+1})$ of the fibre over $l$.  Similarly, the subset $\mathbb{P}(T_l\CP^n)$ of $(T\CP^n)^\mathbb{P}$ is the image under $\mathbb{P}(T\CP^n)\hookrightarrow\mathbb{P}(T\CP^n\oplus\mathbb{C})=(T\CP^n)^\mathbb{P}$ of the fibre over $l$.  However, as in the first paragraph there is a diffeomorphism $\phi\colon\mathbb{P}(\gamma_n^\perp)\to\mathbb{P}(T\CP^n)$ that fits into a commutative diagram
\[\xymatrix{
\mathbb{P}(\gamma_n^\perp)\ar[r]^{\phi}\ar@{^(->}[d]  & \mathbb{P}(T\CP^n)\ar@{^(->}[d]\\
\mathbb{P}(\mathbb{C}^{n+1})\ar[r]_-\phi & \mathbb{P}(T\CP^n\oplus\mathbb{C}),
}\]
so that $\phi$ identifies the required subsets.

Before proving the fourth and fifth claims, we show that the composite embedding
\[\varepsilon\colon T\CP^n\hookrightarrow(T\CP^n)^\mathbb{P}\xrightarrow{\phi^{-1}}\CP^n\times\CP^n\]
is given by  $\varepsilon[\mathbf{u},\mathbf{v}] = ([\mathbf{u}+\mathbf{v}],[\mathbf{u}])$.  The first map above sends $[\mathbf{u},\mathbf{v}]$ to the span of $(\mathbf{v}\otimes\mathbf{u}^\ast,1)\in(\gamma_n^\perp\otimes\gamma_n)\oplus\mathbb{C}$.  We must therefore check that $\phi$ sends $([\mathbf{u}+\mathbf{v}],[\mathbf{u}])$ to this span.  Regarding $\CP^n\times\CP^n$ as $\mathbb{P}(\mathbb{C}^{n+1})$, $([\mathbf{u}+\mathbf{v}],[\mathbf{u}])$ is just the span of $\mathbf{u}+\mathbf{v}$ in the fibre of $\mathbb{C}^{n+1}$ over $[\mathbf{u}]\in\CP^n$.  Applying $\phi$, this is transformed to the span of $\mathbf{v}\otimes\mathbf{u}^\ast + \mathbf{u}\otimes\mathbf{u}^\ast$ in the fibre of $\mathbb{C}^{n+1}\otimes\gamma_n^\ast$ over $[\mathbf{u}]$.  But the isomorphism $\mathbb{C}^{n+1}\otimes\gamma_n^\ast\cong T\CP^n\oplus\mathbb{C}$ sends this vector to $(\mathbf{v}\otimes\mathbf{u}^\ast,1)$ as required.

To prove the fourth claim it suffices to note that $\varepsilon$ is $U(n+1)$-equivariant.   
To prove the final claim we must show that, given a point $[\mathbf{u},\mathbf{v}]\in T\CP^n$,  the embedding $\varepsilon$ sends the pair \eqref{CircleActionEquation} to $t([\mathbf{u}+\mathbf{v}],[\mathbf{u}])=\left([\mathbf{u}+\mathbf{v}],[R(t,[\mathbf{u}+\mathbf{v}])\mathbf{u}]\right)$.  It is not difficult to compute that 
\[R(t,[\mathbf{u}+\mathbf{v}])\mathbf{u}=\frac{{t}+\|\mathbf{v}\|^2}{1+\|\mathbf{v}\|^2}\mathbf{u}+\frac{{t}-1}{1+\|\mathbf{v}\|^2}\mathbf{v}\]
as required.  To complete the proof we need only show that the sum of the two quantities on the right hand side of \eqref{CircleActionEquation} span the line $[\mathbf{u}+\mathbf{v}]$.  But one can compute the sum and verify this directly.
\end{proof}

\subsection{The homology of $\mathbb{L}^n$}\label{LnPinchedComputationsSubsection}
In this subsection we compute the homology of $\mathbb{L}^n$ in terms of the homology of $\mathbb{S}T\CP^n$.  Since Corollary~\ref{LnCorollary} identifies $\mathbb{L}^n$ as $(T\CP^n)^\mathrm{P}$, this is a routine consequence of the results of \S\ref{CompactificationsSection}.

\begin{definition}\label{STCPnModuleDefinition}
Proposition~\ref{LnIsACompactificationProposition} and Corollary~\ref{LnCorollary} show that $(T\CP^n)^\mathrm{P}$ admits actions of $U(n+1)$ and $S^1$.  The first action is obtained from the standard action on $T\CP^n$, so preserves the length of vectors.  One can also check that the second action preserves the length of vectors.  Thus $\mathbb{S}T\CP^n$ inherits actions of $U(n+1)$ and $S^1$.  In this way $H_{\ast-1}(\mathbb{S}T\CP^n)$ becomes a module over $H_\ast(U(n+1))$ and over $H_\ast(S^1)$.  We now twist these module structures, declaring that an element $\alpha$ in $H_\ast(U(n+1))$ or $H_\ast(S^1)$ acts on $H_{\ast-1}(\mathbb{S}T\CP^n)$ as $(-1)^{|\alpha|}$ times the operation defined by the actions above.  We also regard $H_{\ast-1}(\mathbb{S}T\CP^n)$ as a module over $H^\ast(\CP^n)$ but there is no twisting of this module structure.
\end{definition}

\begin{proposition}\label{LnHomologyProposition}
\begin{enumerate}
\item
There is a short exact sequence
\[0\to H_\ast(\CP^n)\to H_\ast(\mathbb{L}^n)\to H_{\ast-1}(\mathbb{S}T\CP^n)\to 0\]
of $H_\ast(U(n+1))$, $H_\ast(S^1)$ and $H^\ast(\CP^n)$ modules.  It is split as a sequence of $H_\ast(U(n+1))$ and of $H^\ast(\CP^n)$ modules, but not as $H_\ast(S^1)$ modules.
\item
The composition $H_\ast(\CP^n\times\CP^n)\xrightarrow{q_\ast}H_\ast(\mathbb{L}^n)\to H_{\ast-1}(\mathbb{S}T\CP^n)$ sends $[\CP^i]\times[\CP^n]$ to $u^{n-i}\cap[\mathbb{S}T\CP^n]$.
\item
Let $R\colon H_\ast(\mathbb{L}^n)\to H_{\ast+1}(\mathbb{L}^n)$ denote the degree-raising operator associated to the $S^1$ action on $\mathbb{L}^n$.
Let $D\in H_1(\mathbb{L}^n)$ be a class whose image in $H_0(\mathbb{S}T\CP^n)$ is the homology class of a point.  Then the composition
\[H_1(\mathbb{L}^n)\xrightarrow{R}H_2(\mathbb{L}^n)\xrightarrow{\pi_\ast}H_2(\CP^n)\]
sends $D$ to $[\CP^1]$.
\end{enumerate}
\end{proposition}
\begin{proof}
By Corollary~\ref{LnCorollary} there is an isomorphism $\psi_\ast\colon H_\ast(\mathbb{L}^n)\cong H_\ast((T\CP^n)^\mathrm{P})$ of $H^\ast(\CP^n)$, $H_\ast(U(n+1))$ and $H_\ast(S^1)$ modules.  The actions of $U(n+1)$ and $S^1$ preserve the lengths of vectors, and the action of $U(n+1)$ commutes with the projection, so the first part is an immediate consequence of Proposition~\ref{PinchedHomologyProposition} and Lemma~\ref{PinchedHomologyGroupActionLemma}.

Let us recall from Proposition~\ref{ProjectiveToPinchedProposition} that the composite $H_\ast((T\CP^n)^\mathbb{P})\to H_\ast((T\CP^n)^\mathrm{P})\to H_{\ast-1}(\mathbb{S}T\CP^n)$ sends $[(T\CP^n)^\mathbb{P}]$ to $[\mathbb{S}T\CP^n]$, sends classes $u_{T\CP^n}\cap x$ to zero, and is a homomorphism of $H^\ast(\CP^n)$ modules.  Now apply the isomorphisms $\phi\colon\CP^n\times\CP^n\to(T\CP^n)^\mathbb{P}$, $\psi\colon\mathbb{L}^n\to(T\CP^n)^\mathrm{P}$ of \S\ref{LnPinchedSubsection}.  The composition just mentioned becomes the composition of part 2 of the proposition; $[(T\CP^n)^\mathbb{P}]$ becomes $[\CP^n]\times[\CP^n]$; $u_{T\CP^n}$ becomes $c_1(\gamma_n\boxtimes\gamma_n^\ast)=1\times u-u\times 1$ by Proposition~\ref{LnIsACompactificationProposition}.  Now
\begin{eqnarray*}
[\CP^i]\times[\CP^n]
&=&(u^{n-i}\times 1)\cap([\CP^n]\times[\CP^n])\\
&=&(1\times u^{n-i})\cap([\CP^n]\times[\CP^n])\\
&&+(u\times 1-1\times u)\cap\lambda\cap([\CP^n]\times[\CP^n])\\
&=&\pi^\ast u^{n-i}\cap([\CP^n]\times[\CP^n])\\
&&+(u\times 1-1\times u)\cap\lambda\cap([\CP^n]\times[\CP^n])
\end{eqnarray*}
where $\lambda=(u^{n-i-1}\times 1+\cdots+1\times u^{n-i-1})$.  By the results of Proposition~\ref{ProjectiveToPinchedProposition} described earlier in the paragraph, this class is sent to $u^{n-i}\cap[\mathbb{S}T\CP^n]$.  This proves the second assertion.

It suffices to prove the third claim for a single choice of $D$.  We will begin by giving an explicit representative for such a class on which we can compute $\pi_\ast\circ R$ directly.  Let $\delta\colon[0,\infty]\to\mathbb{L}^n$ be the map that sends $r$ to the point $([1,r,0,\ldots,0],[1,0,\ldots,0])=([\mathbf{e}_0+r\mathbf{e}_1],[\mathbf{e}_0])$.  Its value at $0$ is $([\mathbf{e}_0],[\mathbf{e}_0])$ and its value at $\infty$ is $([\mathbf{e}_1],[\mathbf{e}_0])$.  These two points are identified in $\mathbb{L}^n$, so $\delta$ is a $1$-cycle in $\mathbb{L}^n$.  One can easily verify that the isomorphism $\psi\colon\mathbb{L}^n\to (T\CP^n)^\mathrm{P}$ of Proposition~\ref{LnIsACompactificationProposition} sends $\delta$ to a cycle $\delta'\colon[0,\infty]\to(T\CP^n)^\mathrm{P}$ with $\delta'(r)=[\mathbf{e}_0,r\mathbf{e}_1]\in T\CP^n$ for $r\in[0,\infty)$.

Now we refer to the proof of Proposition~\ref{PinchedHomologyProposition} where the map $H_\ast((T\CP^n)^\mathrm{P})\to H_{\ast-1}(\mathbb{S}T\CP^n)$ was constructed.  It is clear that the composition of $H_\ast((T\CP^n)^\mathrm{P})\to H_\ast((T\CP^n)^\mathbf{P},\CP^n)$ with the first two isomorphisms of \eqref{ZigZagEquation} sends $[\delta']$ to the homology class of the relative cycle $(I,\partial I)\to (I,\partial I)\times\mathbb{S}T\CP^n$, $r\mapsto (r,[\mathbf{e}_0,\mathbf{e}_1])$.  This is sent by $\Theta\colon H_\ast((I,\partial I)\times\mathbb{S}T\CP^n)\to H_{\ast-1}(\mathbb{S}T\CP^n)$ to the homology class of the point $[\mathbf{e}_0,\mathbf{e}_1]$.  Thus $[\delta]$ is indeed a class of the form $D$ described.

Now one sees immediately that $\pi_\ast\circ R_\ast [\delta]$ is represented by the cycle described in the final paragraph of the proof of Lemma~\ref{LoopHomologyInjective}, and it is shown there that this cycle represents $[\CP^1]$.  This completes the proof.
\end{proof}

\section{Statement of Theorem D and proof of Theorem C}\label{ThirdReformulationSection}

We now state Theorem D, which describes the homology of $\mathbb{S}T\CP^n$, and use it to prove Theorem C.  Homology groups are taken with coefficients in $\mathbb{Z}$ throughout.

\begin{definition}\label{SecondSTCPnModuleDefinition}
Recall from Definition~\ref{STCPnModuleDefinition} that $\mathbb{S}T\CP^n$ is equipped with actions of $U(n+1)$ and $S^1$.  It also admits the projection to $\CP^n$.  In this way $H_{\ast}(\mathbb{S}T\CP^n)$ becomes a module over $H_\ast(U(n+1))$, $H_\ast(S^1)$ and $H^\ast(\CP^n)$.  
Definition~\ref{STCPnModuleDefinition} also explained certain twistings of the module structures, which we \emph{do not} use here.  It is convenient to regard the $H_\ast(S^1)$ module structure in terms of the degree-raising operator $R_\mathbb{S}\colon H_\ast(\mathbb{S}T\CP^n)\to H_{\ast+1}(\mathbb{S}T\CP^n)$ given by applying $[S^1]$.
\end{definition}

\begin{TheoremD}
$H_\ast(\mathbb{S}T\CP^n)$ is generated by classes
\[a_0,\ldots,a_{n-1},\qquad b_0,\ldots,b_n\] 
in degrees $|a_i|=2i$ and $|b_i|=2i+2n-1$, subject to the single relation $(n+1)b_0=0$.  The degree 0 generator $a_0$ is just the homology class of a point.
Furthermore $u\cap a_{i+1}=a_{i}$ and $b_{i+1}=b_{i}$ for $i\geqslant 0$, while $u\cap a_0=0$ and $u\cap b_0=0$.
The action of $H_\ast(U(n+1))$ on $H_\ast(\mathbb{S}T\CP^n)$ is determined by the formulas
\[e_{2j+1}b_n=0,\qquad e_{2j+1}a_{n-1}=-(j+1)b_{j}\]
and $R_\mathbb{S}$ is determined by the formulas
\[R_\mathbb{S} a_i=0\mathrm{\ for\ }0\leqslant i<(n-1),\quad R_\mathbb{S} a_{n-1}=\binom{n+1}{2} b_0,\quad R_\mathbb{S} b_i=0\mathrm{\ for\ }0\leqslant i\leqslant n.\]
\end{TheoremD}

\begin{proof}[Proof of 
Theorem C]
By taking the splitting of the exact sequence of Proposition~\ref{LnHomologyProposition}, there is an isomorphism
\[\Phi\colon H_\ast(\mathbb{L}^n)\cong H_{\ast}(\CP^n) \oplus H_{\ast-1}(\mathbb{S}T\CP^n)\]
of abelian groups.  Taking the standard generators of $H_\ast(\CP^n)$ and the generators of $H_{\ast-1}(\mathbb{S}T\CP^n)$ described in Theorem D, we obtain the required generators and relation.

It is precisely the projection map that provides the splitting in Proposition~\ref{LnHomologyProposition}, so that part 1 follows immediately.

Give $H_{\ast-1}(\mathbb{S}T\CP^n)$ the twisted $H_\ast(U(n+1))$ and $H_\ast(S^1)$ module structures described in Definition~\ref{STCPnModuleDefinition}.  Then, since the sequence of Proposition~\ref{LnHomologyProposition} is a split short exact sequence of $H^\ast(\CP^n)$ and $H_\ast(U(n+1))$ modules, the isomorphism above is an isomorphism of $H^\ast(\CP^n)$ and $H_\ast(U(n+1))$ modules, which using Theorem D gives us part 2 and part 3.

The short exact sequence of Proposition~\ref{LnHomologyProposition} is a short exact sequence of $H_\ast(S^1)$ modules, but is not split.  The degree-raising operator, which we denote by $R$, therefore satisfies $Rx=0$ for $x\in H_\ast(\CP^n)$ and $Ry=-R_\mathbb{S}y\mathrm{\ modulo\ }H_\ast(\CP^n)$ for $y\in H_{\ast-1}(\mathbb{S}T\CP^n)$.  Here $R_\mathbb{S}$ denotes the degree-shifting operator on $H_{\ast}(\mathbb{S}T\CP^n)$.  The description of $R$ in terms of some integers $\mu_{n-1},\ldots,\mu_0$ follows from this remark and Theorem D.  We now note that $a_0$ serves as the class $D$ of Proposition~\ref{LnHomologyProposition}, and that $\pi_\ast\circ R a_0 = [\CP^1]$, so that $\mu_0=1$ by the result of that proposition.

The final part is the conclusion of part 2 of Proposition~\ref{LnHomologyProposition}.
\end{proof}

\section{Proof of Theorem D}\label{STCPnHomologySection}

We have now established that Theorem A follows from Theorem D.  It therefore remains, in this final section, to give the proof of Theorem D.  Recall from Definitions~\ref{STCPnModuleDefinition} and \ref{SecondSTCPnModuleDefinition} that $H_\ast(\mathbb{S}T\CP^n)$ is a module over $H_\ast(S^1)$, $H_\ast(U(n+1))$ and $H^\ast(\CP^n)$.  Theorem D describes $H_\ast(\mathbb{S}T\CP^n)$ and these module structures.  We will prove the theorem by combining a series of lemmas, several of which are instances of the general theory established in \S\ref{SphereBundlesSection}.  

\begin{lemma}\label{STCPnHomologyLemma}
$H_\ast(\mathbb{S}T\CP^n)$ is generated by classes
\[a_0,\ldots,a_{n-1},\qquad b_0,\ldots,b_n\] 
in degrees $|a_i|=2i$ and $|b_i|=2i+2n-1$, subject to the single relation $(n+1)b_0=0$.
Furthermore $u\cap a_{i+1}=a_{i}$ and $b_{i+1}=b_{i}$ for $i\geqslant 0$, while $u\cap a_0=0$ and $u\cap b_0=0$.
The Gysin sequence maps
\[ H_{2i}(\mathbb{S}T\CP^n)\to H_{2i}(\CP^n),\qquad
H_{2i}(\CP^n)\to H_{2i+2n-1}(\mathbb{S}T\CP^n)\]
send $a_i$ to $\CP^i$ and $\CP^i$ to $b_i$ respectively.
\end{lemma}
\begin{proof}
The Gysin sequence \eqref{GysinSequenceEquation} restricts to isomorphisms
$H_i(\mathbb{S}T\CP^n)\cong H_i(\CP^n)$ for $i\leqslant 2n-2$, $H_{i-2n+1}(\CP^n)\cong H_i(\mathbb{S}T\CP^n)$ for $i\geqslant 2n-1$, and a short exact sequence
\[0\to H_{2n}(\CP^n)\xrightarrow{c_n(T\CP^n)\cap -} H_0(\CP^n)\to H_{2n-1}(\mathbb{S}T\CP^n)\to 0.\]
Note that $c_n(T\CP^n)\cap[\CP^n]=(n+1)$.  The description of the groups $H_\ast(\mathbb{S}T\CP^i)$, and of the elements $a_i$, $b_i$, is now immediate.  The effect of taking cap products with $u$ is immediate from Lemma~\ref{GysinCapProductLemma}.
\end{proof}

Let us consider the following inclusions of spaces:
\[ \mathbb{S}T\CP^n \hookrightarrow \mathbb{S}T\CP^{n+1}|\CP^n\hookrightarrow \mathbb{S}T\CP^{n+1}\]
The left-hand space has a natural $U(n+1)$ action, and the right-hand space has a natural $U(n+2)$ action that restricts to a $U(n+1)$ action.  The middle space inherits a $U(n+1)$ action from the right-hand space, and both inclusions are then $U(n+1)$-equivariant.

\begin{lemma}\label{ThreeHomologiesLemma}
$H_\ast(\mathbb{S}T\CP^{n+1}|\CP^{n})$ is free with basis $a_0,\ldots,a_n$, $b_0,\ldots,b_{n}$.  The induced maps
\[ H_\ast(\mathbb{S}T\CP^n) \to H_\ast(\mathbb{S}T\CP^{n+1}|\CP^n)\to H_\ast(\mathbb{S}T\CP^{n+1})\]
are given by $a_i\mapsto a_i,$ $b_0\mapsto 0$, $b_{i+1}\mapsto b_i$ and
$a_i\mapsto a_i$, $b_i\mapsto b_i$ respectively.
\end{lemma}
\begin{proof}
The computation of the homology groups is directly analogous to the computation in Lemma~\ref{STCPnHomologyLemma} but somewhat simpler.  Note that $T\CP^{n+1}|\CP^n=T\CP^n\oplus\gamma_n^\ast$.  The claim about the first map now follows from Lemma~\ref{DirectSumGysinSequenceLemma}.  The description of the second map is immediate from Lemma~\ref{PullbackGysinSequenceLemma}.
\end{proof}

\begin{lemma}\label{ActionCapProductLemma}
Let $n\geqslant 1$ and let $0\leqslant i\leqslant n$.  Then for $x\in H_\ast (\mathbb{S}T\CP^n)$ we have 
\[ u\cap (e_{2i+1}x) = e_{2i+1}(u\cap x).\]
The same formula holds for $x\in H_\ast(\mathbb{S}(T\CP^{n+1}|T\CP^{n})$.
\end{lemma}
\begin{proof}
Let $A\colon U(n+1)\times\mathbb{S}T\CP^n\to\mathbb{S}T\CP^n$ and $A\colon U(n+1)\times\mathbb{S}T\CP^{n+1}|\CP^{n}\to\mathbb{S}T\CP^{n+1}|\CP^{n}$ denote the action maps.  Then in both cases $A^\ast u = 1\times u$.  It suffices to prove this only in the first case, and only for $n\geqslant 2$; the remaining cases follow from this.  To see this in the first case and for $n\geqslant 2$, note that the restriction map $H^2(U(n+1)\times\mathbb{S}T\CP^n)\to H^2(\mathbb{S}T\CP^n)$ is an isomorphism that, since $A$ is an action, sends $A^\ast u$ to $u$.  

Now $u\cap(e_{2i+1}x)=u\cap A_\ast(e_{2i+1}\times x)=A_\ast(A^\ast u \cap (e_{2i+1}\times x))=A_{\ast}((1\times u)\cap (e_{2i+1}\times x))=A_{\ast}(e_{2i+1}\times (u\cap x))=e_{2i+1}(u\cap x)$ as required.
\end{proof}

\begin{lemma}\label{ConstantsExistLemma}
For each $i\geqslant 1$ there is a constant $\lambda_i\in\mathbb{Z}$ such that for all $n\geqslant i$ and $0\leqslant j\leqslant n-1$ the equation
\[ e_{2i+1} a_j = \left\{\begin{array}{cl} \lambda_i b_{i+j-n+1} & i+j\geqslant (n-1)\\ 0 & i+j<(n-1) \end{array}\right.\]
holds in $H_\ast(\mathbb{S}T\CP^n)$.
\end{lemma}
\begin{proof}
First let us suppose that there is a constant $\lambda_i$ such that for each $n\geqslant 1$ the equation
\begin{equation}\label{LambdaEquation}
e_{2i+1}a_{n-1}=\lambda_i b_i
\end{equation}
holds in $H_{2n+2i+1}(\mathbb{S}T\CP^n)$.  Then by Lemma~\ref{ActionCapProductLemma} we have $e_{2i+1}a_j=e_{2i+1}(u^{n-1-j}\cap a_{n-1})=u^{n-1-j}\cap(e_{2i+1}a_{n-1})=\lambda_i u^{n-1-j}\cap b_i=\lambda_{i+j-n+1}b_{i+j-n+1}$ for all $0\leqslant j\leqslant (n-1)$ as required.

We must now find $\lambda_i\in\mathbb{Z}$ such that \eqref{LambdaEquation} holds for all $n\geqslant i$.  Let $\lambda_i$ be determined by the equation $e_{2i+1}a_{i-1}=\lambda_i b_{i}$ in $H_{4i-1}(\mathbb{S}T\CP^i)$.  In other words we define $\lambda_i$ by the instance $n=i$ of \eqref{LambdaEquation}.  We will now show that with this choice of $\lambda_i$ the equation \eqref{LambdaEquation} holds for all $n\geqslant i$.

Fix $n\geqslant i$ and suppose that \eqref{LambdaEquation} holds.  Then by Lemma~\ref{ThreeHomologiesLemma} we have $e_{2i+1}a_{n-1}=\lambda_i b_{i-1}$ in $H_{2n+2i-1}(\mathbb{S}T\CP^{n+1}|\CP^n)$.  There is some constant $\mu$ such that $e_{2i+1}a_n=\mu b_i$ in $H_{2n+2i+1}(\mathbb{S}T\CP^{n+1}|\CP^n)$.  By taking the cap product with $u$ on both sides of this equation and applying Lemma~\ref{ActionCapProductLemma} we find $\mu=\lambda_i$.  Thus $e_{2i+1}a_n=\lambda_i b_i$ in $H_{2n+2i+1}(\mathbb{S}T\CP^{n+1}|\CP^n)$.  Applying Lemma~\ref{ThreeHomologiesLemma} again we find that $e_{2i+1}a_n=\lambda_i b_i$ in $H_{2n+2i+1}(\mathbb{S}T\CP^{n+1})$.  This is equation \eqref{LambdaEquation} with $n$ replaced by $(n+1)$.

Equation~\eqref{LambdaEquation} therefore holds for all $n$ by induction, and this completes the proof.
\end{proof}

\begin{lemma}\label{PlusMinusNPlusOneLemma}
$e_{2n+1}a_0=\pm(n+1)b_1$ in $H_\ast(\mathbb{S}T\CP^n)$.
\end{lemma}
\begin{proof}
Recall that we have described $T\CP^n$ as the space
\[\left\{(\mathbf{u},\mathbf{v})\in\mathbb{C}^{n+1}\times\mathbb{C}^{n+1}\mid\|\mathbf{u}\|=1,\langle \mathbf{u},\mathbf{v}\rangle=0\right\}\big\slash S^1\]
where $S^1$ acts by the formula $t(\mathbf{u},\mathbf{v})=({t}\mathbf{u}, {t}\mathbf{v})$.   We can therefore describe $\mathbb{S}T\CP^n\subset T\CP^n$ as the space $F_n/S^1$, where $F_n=\{(\mathbf{u},\mathbf{v})\in\mathbb{C}^{n+1}\times\mathbb{C}^{n+1}\mid\|\mathbf{u}\|=1,\|\mathbf{v}\|=1,\langle \mathbf{u},\mathbf{v}\rangle=0\}$ is the Stiefel manifold $U(n+1)/U(n-1)$.  

Recall that $H_\ast(U(n+1))=\Lambda_\mathbb{Z}(e_1,\ldots,e_{2n+1})$.  It is well-known that as a module over this ring $H_\ast(F_n)=\Lambda_\mathbb{Z}(e_1,\ldots,e_{2n+1})/\langle e_1,\ldots,e_{2n-1}\rangle$, so that its homology groups are nonzero only in degrees $0,2n-1,2n+1,4n$.  The class $e_{2n+1}a_0$ is the image of $e_{2n+1}$ under $H_\ast(F_n)\to H_\ast(\mathbb{S}T\CP^n)$.

Let us consider the Serre spectral sequence $(E^r,d^r)$ of the fibration $F_n\to \mathbb{S}T\CP^n\to BS^1$.  This has $E^2_{\ast,\ast}=H_\ast(BS^1)\otimes H_\ast(F_n)$.  For each $i\geqslant 0$ we write $a_i$ for a generator of $H_{2i}(BS^1)$.  For degree reasons we have $E^2=E^3=\cdots =E^{2n}$.  From our computation of $H_\ast(\mathbb{S}T\CP^n)$ in Lemma~\ref{STCPnHomologyLemma} we know that $H_{2n-1}(\mathbf{S}T\CP^n)=\mathbb{Z}/(n+1)$, and this forces $d^{2n}(a_n\otimes 1)=\pm(n+1)1\otimes e_{2n-1}$.  It follows that $d^{2n}(a_{n+i}\otimes 1)=\pm(n+1)a_i\otimes e_{2n-1}$ for $i\geqslant 0$.  No further differentials effect the groups in total degree $(2n+1)$.

From the last paragraph we find that $E^\infty_{0,2n+1}=\mathbb{Z}$ with generator $1\otimes e_{2n+1}$, that $E^\infty_{2,2n-1}=\mathbb{Z}/(n+1)$, and that all other terms of $E^\infty$ in total degree $2n+1$ vanish.  But $H_{2n+1}(\mathbb{S}T\CP^n)$ is a copy of $\mathbb{Z}$ with generator $b_0$, and so $E^\infty_{0,2n+1}\to H_{2n+1}(\mathbb{S}T\CP^n)$ must send $1\otimes e_{2n+1}$ to $\pm(n+1)b_0$.  By the first paragraph, $1\otimes e_{2n+1}$ represents the element $e_{2n+1}a_0$, and so this completes the proof.
\end{proof}

\begin{lemma}\label{ETwoNMinusOneLemma}
$e_{2n-1}a_0=b_0$ in $H_\ast(\mathbb{S}T\CP^n)$.
\end{lemma}
\begin{proof}
Consider the inclusion $\mathbb{S}T\CP^n|\CP^0\hookrightarrow \mathbb{S}T\CP^n$.  If we regard $U(n)\subset U(n+1)$ as the subgroup consisting of transformations that leave the first coordinate unchanged, then the action of $U(n)$ on $\mathbb{S}T\CP^n$ restricts to an action on $\mathbb{S}T\CP^n|\CP^0$.  Indeed, under the identification of $\mathbb{S}T\CP^n|\CP^0$ with the sphere in $\mathbb{C}^n$ this is just the standard action.

Now $H_\ast(\mathbb{S}T\CP^n|\CP^0)$ is free on generators $a_0$ and $b_0$ that are mapped to the elements of the same name in $H_\ast(\mathbb{S}T\CP^n)$ (compare with Lemma~\ref{STCPnHomologyLemma} or Lemma~\ref{ThreeHomologiesLemma}).  We can therefore prove the lemma by showing that $e_{2n-1}a_0=b_0$ in $H_\ast(\mathbb{S}T\CP^n|\CP^0)$.

Recall the rotation map $R\colon S^1\times\CP^{n-1}\to U(n)$ from Definition~\ref{RotationDefinition} and consider the map $U\colon S^1\times\CP^{n-1}\times S^{2n-1}\to \CP^{n-1}\times S^{2n-1}$ that sends $(t,l,\mathbf{v})$ to $(l,R(t,l)\mathbf{v})$.  By the definition of $e_{2n-1}$ we have $e_{2n-1}a_0={\pi_2}_\ast U_\ast([S^1]\times[\CP^{n-1}]\times\mathrm{pt})$.  We must therefore show that $U_\ast([S^1]\times[\CP^{n-1}]\times\mathrm{pt})=\mathrm{pt}\times[S^{2n-1}]$.  But $U([S^1]\times -)$ is the degree-raising operator associated to the action of $S^1$ on
\[\CP^{n-1}\times S^{2n-1}=\mathbb{S}(\CP^{n-1}\times\mathbb{C}^n)=\mathbb{S}(\gamma_{n-1}^\perp\oplus\gamma_{n-1})\]
given by scalar multiplication in the second summand.  By Theorem~\ref{SphereBundlesTheorem}, $U_\ast([S^1]\times -)$ is therefore equal to the composite
\begin{multline*}
H_{2n-2}(\CP^{n-1}\times S^{2n-1})\xrightarrow{\pi_\ast}H_{2n-2}(\CP^{n-1})
\xrightarrow{E_{\gamma_{n-1}^\perp}\cap -}\\ H_{0}(\CP^{n-1})\to H_{2n-1}(\CP^{n-1}\times S^{2n-1})
\end{multline*}
so that it suffices to show that $\langle E_{\gamma_{n-1}^\perp},[\CP^{n-1}]\rangle =1$.  But $c(\gamma_{n-1}^\perp)=c(\gamma_{n-1})^{-1}=(1-u)^{-1}=1+u+\cdots+u^{n-1}$, so that $\langle E_{\gamma_{n-1}^\perp},[\CP^{n-1}]\rangle =1$ as required.
\end{proof}

\begin{lemma}\label{EOneLemma}
$e_1 a_{n-1}=-b_0$ in $H_\ast(\mathbb{S}T\CP^n)$.
\end{lemma}
\begin{proof}
Consider the actions
\begin{align*}
A\colon& S^1\times\mathbb{S}T\CP^n\to\mathbb{S}T\CP^n, \\
B\colon& S^1\times\mathbb{S}T\CP^n|\CP^{n-1}\to\mathbb{S}T\CP^n|\CP^{n-1}.
\end{align*}
obtained by scalar multiplication in the final coordinate.  By the definition of $e_1$ (Definition~\ref{RotationDefinition}) we have $e_1a_{n-1}=A_\ast([S^1]\times a_{n-1})$.  We claim that $B_\ast([S^1]\times a_{n-1})=n b_0$.  This will prove the lemma, because then $A_\ast([S^1]\times a_{n-1})=n b_0=-b_0$.  (Recall that $H_{2n-1}(\mathbb{S}T\CP^n)=\mathbb{Z}/(n+1)$.)

Note that $T\CP^n|\CP^{n-1}=T\CP^{n-1}\oplus\gamma_{n-1}^\ast$, and with respect to this splitting $B$ is just the action given by scalar multiplication in the second summand.  The operator $B_\ast([S^1]\times -)$ is then the associated degree-raising operator, which by Theorem~\ref{SphereBundlesTheorem} is equal to the composite
\begin{multline*}
H_{2n-2}(\mathbb{S}T\CP^n|\CP^{n-1})\to H_{2n-2}(\CP^{n-1})\xrightarrow{E_{T\CP^{n-1}}\cap-}\\
H_0(\CP^{n-1})\to H_{2n-1}(\mathbb{S}T\CP^n|\CP^{n-1}).
\end{multline*}
This map sends $a_{n-1}$ to $\langle E_{T\CP^{n-1}},[\CP^{n-1}]\rangle b_0=n b_0$ as required.
\end{proof}

\begin{lemma}\label{CircleActionsHomotopyLemma}
The action of $S^1$ that $\mathbb{S}T\CP^n$ inherits from $(T\CP^n)^\mathrm{P}$ is homotopic to the action given by scalar multiplication in the fibres.
\end{lemma}
\begin{proof}
By Proposition~\ref{LnIsACompactificationProposition} the first action is given by the formula
\[t[\mathbf{u},\mathbf{v}]=\left[\frac{t+1}{2}\mathbf{u}+\frac{t-1}{2}\mathbf{v},\frac{t-1}{2}\mathbf{u}+\frac{{t}+1}{2}\mathbf{v}\right]\]
while the second action is given by $t[\mathbf{u},\mathbf{v}]=[\mathbf{u},t\mathbf{v}]$.  A homotopy $h\colon[0,{1/2}]\times S^1\times\mathbb{S}T\CP^n\to \mathbb{S}T\CP^n$ between the two is given by sending $(\mu,t,[\mathbf{u},\mathbf{v}])$ to
\[\left[
\left(\mu t+(1-\mu)\right)\mathbf{u}+\sqrt{\mu(1-\mu)}(t-1)\mathbf{v},
\left(t(1-\mu)+\mu\right)\mathbf{v}+\sqrt{\mu(1-\mu)}(t-1)\mathbf{u}
\right].\]
\end{proof}

\begin{proof}[Proof of Theorem D]
The first part of the theorem, which describes the groups $H_\ast(\mathbb{S}T\CP^n)$, follows directly from Lemma~\ref{STCPnHomologyLemma}.  The rest of the theorem gives formulas describing the various module structures on $H_\ast(\mathbb{S}T\CP^n)$.  Those formulas that state that a certain quantity vanishes are all true for degree reasons.  We must prove the remaining two formulas.

Let us show that $e_{2j+1}a_{n-1}=-(j+1)b_{j}$ in $H_\ast(\mathbb{S}T\CP^n)$.  In the case $j=0$ this is the result of Lemma~\ref{EOneLemma}.  For $j\geqslant 1$ we must show that the constant $\lambda_j$ whose existence is guaranteed by Lemma~\ref{ConstantsExistLemma} is equal to $-(j+1)$.  Lemma~\ref{PlusMinusNPlusOneLemma} implies that $\lambda_j=\pm(j+1)$, while Lemma~\ref{ETwoNMinusOneLemma} implies that $\lambda_{j-1}=1$ modulo $(j+1)$.  These together imply that $\lambda_j=-(j+1)$ for all $j$, as required.

We now show that $R_\mathbb{S} a_{n-1}=\binom{n+1}{2} b_0$.  Recall that $R_\mathbb{S}$ is defined using the action of $S^1$ on $\mathbb{S}T\CP^n$ inherited from the action of $S^1$ on $T\CP^n$ described in Proposition~\ref{LnIsACompactificationProposition}.  By Lemma~\ref{CircleActionsHomotopyLemma} this is homotopic to the action given by scalar multiplication in each fibre.  Theorem~\ref{SphereBundlesTheorem} then states that $R_\mathbb{S}$ is the composite
\[H_\ast(\mathbb{S}T\CP^n)\xrightarrow{\pi_\ast} H_\ast(\CP^n)\xrightarrow{c_{n-1}(T\CP^n)}H_{\ast-2n+2}(\CP^n)\to H_{\ast+1}(\mathbb{S}T\CP^n)\]
so that by the properties of $a_{n-1}$ and $b_0$ listed in Lemma~\ref{STCPnHomologyLemma} we need only show that $\langle c_{n-1}(T\CP^n),[\CP^{n-1}]\rangle=\binom{n+1}{2}$.  But this is immediate from the fact that $c(T\CP^n)=(1+u)^{n+1}$ and that $\langle u^{n-1},[\CP^{n-1}]\rangle=1$.
\end{proof}

\bibliographystyle{alpha}
\bibliography{StringTopologyBibliography}
\end{document}